\documentclass[11pt,a4paper,leqno]{amsart}
\usepackage[latin1]{inputenc}
\usepackage[T1]{fontenc}
\usepackage{amsfonts}

\usepackage{amsmath}
\usepackage{amssymb}
\usepackage{eurosym}
\usepackage{mathrsfs}
\usepackage{palatino}
\usepackage{pxfonts}
\usepackage{color}
\usepackage{esint}
\usepackage{url}
\usepackage{verbatim}
\usepackage{mathtools}
\mathtoolsset{showonlyrefs}

\usepackage{enumerate}
\usepackage[foot]{amsaddr}

\usepackage[pagebackref,hypertexnames=false, colorlinks, citecolor=blue, linkcolor=blue, urlcolor=red]{hyperref}
\usepackage[nobysame]{amsrefs}

\newcommand{\Pp}{\mathbb{P}}

\newcommand{\R}{\mathbb{R}}
\newcommand{\e}{\mathrm{e}}

\newcommand{\T}{\mathsf{T}}

\newcommand{\dist}[0]{\operatorname{dist}}

\newcommand{\supp}[0]{\operatorname{supp}}

\newcommand{\sign}[0]{\operatorname{sgn}}

\newcommand{\eps}[0]{\varepsilon}

\newcommand{\UMD}{\operatorname{UMD}}

\swapnumbers
\theoremstyle{plain}
\newtheorem{thm}[equation]{Theorem}
\newtheorem{lem}[equation]{Lemma}
\newtheorem{prop}[equation]{Proposition}
\newtheorem{cor}[equation]{Corollary}

\theoremstyle{definition}

\theoremstyle{remark}
\newtheorem{rem}[equation]{Remark}

\numberwithin{equation}{section}

\author[F. Di Plinio]{Francesco Di Plinio}
\address[F. Di Plinio]{Department of Mathematics, Washington University in St. Louis, One Brookings Drive,  St. Louis,  MO 63130-4899, USA}
\email{francesco.diplinio@wustl.edu}

\author[K.\ Li]{Kangwei Li}
\address[K.\ Li]{Center for Applied Mathematics, Tianjin University, Weijin Road 92, 300072 Tianjin, China}
\email{kangwei.nku@gmail.com}

\author[H.\ Martikainen]{Henri Martikainen}
\address[H.\ Martikainen]{Department of Mathematics and Statistics, University of Helsinki, P.O.B. 68, FI-00014 University of Helsinki, Finland}
\email{henri.martikainen@helsinki.fi}

\author[E. Vuorinen]{Emil Vuorinen}
\address[E. Vuorinen]{Centre for Mathematical Sciences, University of Lund, P.O.B. 118, 22100 Lund, Sweden}
\email{j.e.vuorin@gmail.com}

\title[Banach-valued modulation invariant singular integrals]{Banach-valued multilinear singular integrals \\ with modulation invariance}

\makeatletter
\@namedef{subjclassname@2010}{%
  \textup{2010} Mathematics Subject Classification}
\makeatother

\thanks
{F. Di Plinio has been
 partially supported by the National Science Foundation under the grant   NSF-DMS-1800628.
\\
H. Martikainen was supported by the Academy of Finland through the grants 294840, 306901 and  327271  and by the three-year research grant 75160010 of the University of Helsinki.
He is a member of the Finnish Centre of Excellence in Analysis and Dynamics Research supported by the Academy of Finland (project No. 307333).
\\
E. Vuorinen was supported  by the
Jenny and Antti Wihuri Foundation.}

\begin{document}

\begin{abstract}
We prove that the class of   trilinear multiplier forms with singularity over a one dimensional subspace, including the bilinear Hilbert transform, admit bounded  $L^p$-extension to triples of intermediate $\UMD$ spaces. No other assumption, for instance of Rademacher maximal function type, is made on the triple of $\UMD$ spaces.
Among the  novelties in our analysis is an extension of the phase-space projection technique to the $\mathrm{UMD}$-valued setting. This is then employed to obtain appropriate single tree estimates by appealing to the  $\mathrm{UMD}$-valued bound for bilinear Calder\'on-Zygmund operators recently obtained  by the same authors.  \end{abstract}
\maketitle

\section{Introduction and main results}

Let ${\mathcal{X}_k}, k=1,2,3, $ be Banach spaces with a trilinear contraction ${\mathcal{X}_1} \times {\mathcal{X}_2} \times {\mathcal{X}_3} \to \mathbb C$ which we denote by $(e_1,e_2,e_3) \mapsto e_1e_2e_3=\prod_{k=1}^3 e_k$. To a multiplier $m$ defined on the orthogonal complement $\Gamma$ of $(1,1,1)\in \R^3$, we may associate the trilinear form
\begin{equation}
\label{e:Lambdam}
\Lambda_m (f_1,f_2,f_3) = \int_\Gamma m(\xi) \left(\prod_{k=1}^3 \widehat f_k(\xi) \right) \mathrm{d} \xi
\end{equation}
acting  on functions $f_k \in \mathcal S(\R)\otimes {\mathcal{X}_k}, k=1,2,3,$ where the former is the Schwartz class. This article is concerned with multipliers $m$ whose  singularity lies on a one-dimensional subspace perpendicular to   a unit vector $\beta \in \Gamma$ which is nondegenerate in the sense that
\begin{equation}
\label{e:nondeg}
\Delta(\beta) \coloneqq \min_{j\neq k} |\beta_j-\beta_k|>0,
\end{equation}
and satisfies for all multi-indices $\alpha$
\begin{equation}
\label{decay}
 \sup_{ \xi \in\Gamma }\big(  \dist(  \xi,  \beta ^\perp)\big)^\alpha \big| \partial_\alpha m ({\xi})\big| \lesssim_\alpha 1.
\end{equation}
Assumption \eqref{decay} is a $\beta^\perp$-modulation invariant version of the Coifman-Meyer condition.
 This class includes the bilinear Hilbert transform with parameter $\beta$, whose dual trilinear multiplier form may be obtained by choosing
\[
m(\xi)= \sign(\xi \cdot \beta).
\]
The bilinear Hilbert transform
\[
T(f_1,f_2)(x)= \int f_1(x-b_1 t) f_2(x-b_2 t) \, \frac{\mathrm{d} t}{t}, \qquad b_j=\beta_j-\beta_3, \, j=1,2,
\]
was first introduced by Calder\'on within the context of $L^p$ estimates for the first commutator of the Cauchy integral along Lipschitz curves. The celebrated articles of Lacey and Thiele \cite{LT1,LT2} contain the first proof of $L^p$ estimates for the bilinear Hilbert transform, while    more general multipliers of the  class \eqref{decay} were treated by Muscalu, Tao and Thiele \cite{MTT}.
\subsection{Main results} In this article, we prove that the trilinear multiplier forms  \eqref{e:Lambdam}, where $m$ is a  multiplier of the class \eqref{decay}, admit  $L^p$-bounded extensions to  triples of intermediate  $\UMD$ Banach spaces. This terminology has been introduced by Hyt\"onen and Lacey \cite{HLC,HLW}: we repeat this definition below and send to \cite{HNVW1} and references therein for background and generalities on $\UMD$ Banach spaces.

 Let $ 2\leq q \leq \infty$ and  $\mathcal X_0, {\mathcal{X}_1}$ be a couple of compatible Banach spaces, { with $\mathcal X_0$ being a $\UMD$ space and $\mathcal X_1$ being a Hilbert space. We say that the   Banach space $\mathcal X$ is $q$-\emph{intermediate} $\UMD$  if \[\mathcal X=[\mathcal X_0,{\mathcal{X}_1}]_{\frac2q}\] namely, $\mathcal X$ is the complex interpolation of a $\UMD$ Banach space with a Hilbert space. Such Banach space $\mathcal X$ is automatically a $\UMD$ space. Notice that  $\mathcal X$ is $q$-intermediate $\UMD$ if and only if its Banach dual $\mathcal X'$ is.

The precise statement of our main result is as follows.
\begin{thm} \label{t:main}
Let $\mathcal X_{j}, j=1,2,3,$ be Banach spaces with Banach duals $\mathcal Y_j=\mathcal X_j'$  and suppose that each $\mathcal X_j$ is $q_{\mathcal X_j}$-intermediate $\UMD$. Assume that
\begin{equation}
\label{e:inttype}
\rho=\sum_{j=1}^3 \frac{1}{q_{\mathcal X_j}} -1 > 0.
\end{equation} Let $\sigma$ be any   permutation of $\{1,2,3\}$,
$m $ be a multiplier satisfying \eqref{decay} and $T_{m,\sigma}$ denote the adjoint bilinear operator to \eqref{e:Lambdam} acting on pairs of  $\mathcal X_{\sigma(1)},\mathcal X_{\sigma(2)}$ functions. Then
\[
\left\|T_{m,\sigma}(f_1,f_2)\right\|_{L^{\frac{p_1p_2}{p_1+p_2}}(\R; \mathcal Y_{\sigma(3)})}
\lesssim \left\|f_{\sigma(1)} \right\|_{L^{p_1}(\R; \mathcal X_{\sigma(1)}) }
\left\|f_{\sigma(2)} \right\|_{L^{p_2}(\R; \mathcal X_{\sigma(2)}) }
\]
whenever
\begin{equation}
\label{e:range}
1<p_1,p_2\leq \infty, \qquad (p_1,p_2) \neq (\infty, \infty), \qquad \left(\frac{1}{p_1},\frac{1}{p_2}\right) \in \mathrm{int}(\mathcal H).
\end{equation}
Here
 $\mathcal H$ is the hexagon  with vertices $A,B,C,D,E,F$  as follows:
 \begin{align*}
	A:&\quad\left(\textstyle \frac{1}{q_{\mathcal X_1}}-\rho q_{\mathcal X_3},\frac{1}{q_{\mathcal X_2}} \right),\,\qquad\qquad\qquad D:\quad \left(\textstyle\frac{1}{q_{\mathcal X_1}}+\rho q_{\mathcal X_1}-\rho,\frac{1}{q_{\mathcal X_2}} \right),
	\\
	B:&\quad\left(\textstyle\frac{1}{q_{\mathcal X_1}},\frac{1}{q_{\mathcal X_2}}-\rho q_{\mathcal X_3} \right),\,\qquad\qquad\qquad E:\quad \left(\textstyle\frac{1}{q_{\mathcal X_1}},\frac{1}{q_{\mathcal X_2}}+\rho q_{\mathcal X_2}-\rho \right),
	\\
	C:&\quad \left(\textstyle\frac{1}{q_{\mathcal X_1}}+\rho q_{\mathcal X_1}-\rho, \frac{1}{q_{\mathcal X_2}}-\rho q_{\mathcal X_1} \right),\quad\quad F:\quad \left(\textstyle\frac{1}{q_{\mathcal X_1}}-\rho q_{\mathcal X_2}, \frac{1}{q_{\mathcal X_2}}+\rho q_{\mathcal X_2}-\rho\right).
\end{align*}

\end{thm}
We note in passing that if condition \eqref{e:inttype} holds, the range $\mathrm{int}(\mathcal H)$ is nonempty and in particular contains the region
\[
 q_{\mathcal X_k}<p_k<\infty,\; k=1,2,3, \qquad
p_3\coloneqq\left(\textstyle \frac{p_1p_2}{p_1+p_2} \right)'
\]
which is the analogue of  the local $L^2$ range for the  scalar case, see \cite{LT1}. In addition, we point out  that $\mathrm{int}(\mathcal H)$ may  contain quasi-Banach pairs $(p_1,p_2)$, that is, pairs with $\frac{p_1p_2}{p_1+p_2}<1$. This is easier to see by particularizing  Theorem \ref{t:main} to the case
\[
\mathcal X_1= \mathcal X,\, \mathcal X_2=\mathcal X', \,\mathcal X_3=\mathbb C,
\] as in the following corollary. Herein, quasi-Banach estimates are available if $2<q<3$.
\begin{cor} \label{cor:X} Let $\mathcal X$ be a  $q$-intermediate $\UMD$ space and define the trilinear contraction
\[
(\mathfrak{x},\phi,\lambda) \in \mathcal X \times \mathcal X' \times \mathbb C \mapsto \lambda\phi(\mathfrak{x}).
\]
Let
$m $ be a multiplier satisfying \eqref{decay} and $T_{m}$ denote the adjoint bilinear operator to \eqref{e:Lambdam} acting on pairs of  $ \mathcal X , \mathcal X '$ functions.

Suppose that $2\leq q\leq 3$. Then
\begin{equation}
\label{e:corX2}
\left\|
T_{m}(f_1,f_2) \right\|_{L^{\frac{p_1p_2}{p_1+p_2}}(\R)}
\lesssim \left\|f_1 \right\|_{L^{p_1}(\R; \mathcal X) }
\left\|f_{\sigma(2)} \right\|_{L^{p_2}(\R;\mathcal X') }
\end{equation}
whenever
\begin{equation}
\label{e:corX3}
1+\frac{(q-1)(q-2)}{q(5-q)-2}<p_1,p_2\leq \infty, \qquad  \frac23\left(1+\frac{q-2}{5-q} \right)<\frac{p_1p_2}{p_1+p_2} < \frac{q}{q-2}.
\end{equation}
If $3<q<4 $, then \eqref{e:corX2} holds true if, in addition to \eqref{e:corX3}, the condition
\begin{equation}
\label{e:corX4}
\frac{q^2-3q+1}{q}<\min_{(u,v)\in\{(1,2), (2,1)\}}\left\{
\frac{q-1}{p_u} +\frac{q-2}{p_v}
\right\}\end{equation}
is verified.
\end{cor}
Theorem \ref{t:main} and Corollary \ref{cor:X} further the rather recent line of research on the  extension  of singular operators with modulation invariance properties  to   $\UMD$ Banach spaces \emph{without any UMD lattice function space structure,} or lattice structure altogether: a prototypical example are noncommutative $L^p$ spaces such as the reflexive Schatten-von Neumann subclasses of the algebra of bounded operators on a Hilbert space. This line of research was initiated by Hyt\"onen and Lacey in their proof of boundedness of the Carleson maximal partial Fourier sum operator for intermediate $\UMD$ spaces   in the Walsh \cite{HLW} and Fourier \cite{HLC} setting; see also \cite{HLPC} for Walsh-Carleson variation norm bounds. Subsequently, the same authors and Parissis \cite{HLP} proved the analogue of Theorem \ref{t:main} for the Walsh model of the bilinear Hilbert transform. In fact, the range of exponents $\mathrm{int}(\mathcal H)$ is the same as the one obtained therein for the Walsh model, see \cite[Theorem 9.3]{HLP}.
 Results in the vein of \cite{HLP} were recently reproved by Amenta and Uraltsev in \cite{AU1} as a byproduct of novel Banach-valued outer $L^p$ space embeddings for the Walsh wave packet transform.

The theory of $\UMD$-valued linear singular integrals of Calder\'on-Zygmund type is rooted in the works by Burkholder  \cite{Burk1} and Bourgain \cite{Bou} among others, and has been extensively developed since then, see for instance \cite{JX,KW1,HH,HytPor08,Hyt2005,CH,Weis,GMSS} and the monograph \cite{HNVW1}.  Recent advances have concerned the $\UMD$ extension of multilinear Cald\'eron-Zygmund operators \cite{DLMV1,DLMV2,DO}. The above mentioned references deal with generic $\UMD$ spaces, as opposed to lattices,  and thus develop  fundamentally different  techniques from those of the classical vector-valued theory of e.g.\ Benedek, Calder\'on and Panzone, Fefferman-Stein, Rubio de Francia, which are strictly tied to $A_p$-type weighted norm inequalities.
In a similar contrast, the present article combines novel technical tools in $\UMD$-valued time frequency analysis to the  $\UMD$ interpolation space idea of \cite{HLC}   in order to deal  with \emph{multilinear} modulation invariant operators  on  \emph{non-lattice} $\UMD$ spaces, which are out of reach for typical lattice-based techniques.

Nevertheless, a systematic function space-valued theory for \eqref{e:Lambdam} is quite recent.  The first proof of $\ell^p$-valued bounds for the bilinear Hilbert transform in a wide range of exponents is due to P.\ Silva \cite{Silva12}. In \cite{Silva12}, those estimates have been employed to obtain bounds for the biparameter bilinear operator obtained by tensoring  the bilinear Hilbert transform with a Coifman-Meyer multiplier. Several extensions and refinements of \cite{Silva12} have since appeared, see  e.g. \cite{ALV,BenMusc15,CrM,CDPO}.
In general, as Corollary \ref{cor:X} demonstrates, Theorem \ref{t:main} is outside the scope of the above references, although it does imply a strict subset  of the $\ell^p$ estimates of \cite{Silva12}. We send to \cite{AU1,HLP} for a detailed discussion of this point.

However, to stress the difference with the results of \cite{Silva12} and followups, we would like to showcase here a further application of Theorem \ref{t:main} to a triple of non-function, non-lattice $\UMD$ Banach spaces, in addition to that of Corollary
\ref{cor:X}.
In the corollary that follows we  denote   by $S^p$, $1\leq p<\infty$ the $p$-th Schatten-von Neumann class, namely, the subspace of the von Neumann algebra $ \mathcal B(H)$ of linear bounded operators on a separable Hilbert space defined by the norm
\[\|A\|_{S^p} = \|s_n(A)\|_{\ell^p(n \in \mathbb N)}
\]
where $\{s_n(A):n \in \mathbb N\}$ is the sequence of singular values of $A$, that is eigenvalues of the Hermitian operator $|A|=\sqrt{A^*A}$. Notice that the classes $S^p$ are increasingly nested with $p$ and that the trilinear form
\[
(A_1,A_2,A_3) \in S^{t_1} \times S^{t_2} \times S^{t_3}  \mapsto \mathrm{trace}(A_1 A_2 A_3)
\]
is a contraction provided that
\begin{equation}
\label{superholder} \sum_{k=1}^3 \frac{1}{t_k} \geq 1.
\end{equation}
\begin{cor} \label{cor:sc} Suppose that the exponents $1<t_1,t_2,t_3<\infty$ satisfy, in addition to \eqref{superholder}, the conditions
\begin{equation}
\label{e:schat}
 \rho=\sum_{k=1}^3 \frac{1}{\max\{t_k,(t_k)'\}}-1 > 0.
\end{equation}
Let  $\sigma$ be a permutation  of $\{1,2,3\}$, and
$m $ be  a multiplier satisfying \eqref{decay}
Then the corresponding adjoint bilinear operator $T_{m,\sigma}$ maps
\[
T_{m,\sigma}: {L^{p_1}(  \R; S^{t_{\sigma(1)}}) } \times {L^{p_2}(  \R; S^{t_{\sigma(2)}}) } \to L^{\frac{p_1p_2}{p_1+p_2}}(\R; S^{t_{\sigma(3)}'})
\]
boundedly for all
 $p_1,p_2$ specified by \eqref{e:range}.
\end{cor}
Corollary \ref{cor:sc} is obtained from Theorem \ref{t:main} by noticing that $S^p$, $1<p<\infty$ is intermediate $\UMD$ of exponent $\max\{p,p'\}$.
  Similar statements may be obtained for more general non-commutative $L^p$ spaces.  We   send to \cite{PX} for comprehensive definitions and background: a quick harmonic analyst-friendly introduction is given in \cite[Section 3]{DLMV1}.

\subsection{Techniques of proof and novelties} The standard proofs of  $L^p$-bounds for the scalar-valued versions of the forms $\Lambda_m$ in \eqref{e:Lambdam} are articulated in roughly three separate moments. The first is to realize that the forms \eqref{e:Lambdam} lie in the convex hull of suitable discretized model versions, the so-called \emph{tri-tile forms}, displaying the same modulation and translation invariance properties of the condition \eqref{decay}: this step extends \emph{verbatim} to the vector-valued case. We may thus focus on $L^p$-bounds for the model sums.

An essential step of the proof is the decomposition of the model operators into (discretized) multipliers which are adapted to a certain fixed top frequency and localized in space to a top interval. These \emph{tree model sums} are essentially trilinear Calder\'on-Zygmund forms. The contribution of each tree is then controlled by localized space-frequency norms of the involved functions, the so-called \emph{energies} (or sizes). This bound is referred to as \emph{tree estimate}.

In the vector-valued case, this second step has to be adapted in a nontrivial and novel fashion. First of all, the vector-valued \emph{energies}, introduced in \eqref{e:size}, \eqref{e:engdef} must be defined in terms of local $q$-norms of (linear) tree operators rather than simply $\ell^2$ sums of wavelet coefficients coming from each tree. We do so by means of a technical modification of the approach in \cite{HLC}. Second, and most important, we obtain an effective tree estimate by replacing the involved functions with  vector-valued phase-space projections to the space-frequency support of the tree. This extension of the scalar-valued phase-space projections of e.g.\
\cite{DT,MTT2} to $\UMD$ spaces, which may be of independent interest, is carried out in Proposition \ref{p:psp}, and is the main technical novelty of the article.    The tree model sum acts on the phase-space projections roughly as a trilinear CZ multiplier operator, and the $L^p$-norms of the constructed projections are controlled by the corresponding energies. These observations may be used in conjunction with the $L^p$-bound for $\UMD$ extensions of bilinear CZ operators, recently obtained  by the authors of this paper in \cite{DLMV1},  to produce   the tree estimate of Lemma \ref{l:treeest}.

Finally, the recomposition  of the bounds obtained for each tree into a global estimate relies on almost-orthogonality considerations. To export this almost-orthogonality   to the vector-valued scenario, we rely, as in previous literature \cite{HLC,HLP,AU1}, on the $q$-intermediate property of the involved spaces $\mathcal X$.  This step is carried out  in Lemma \ref{l:energy}. As every known example of $\UMD$ space is $q$-intermediate for some $q$,  this assumption may seem  harmless. However,  unlike the linear setting of \cite{HLC},  it is the combined $q$-intermediate type of the three spaces that  introduces the restriction \eqref{e:inttype} and influences the range $\mathrm{int}(\mathcal H)$ in Theorem \ref{t:main}. Further investigation on the necessity and on possible weakening of the  $q$-intermediate assumptions are left for future work.

\subsection*{Plan of the paper} Section 2 contains the preliminary material needed to define the model tri-tile forms. Section \ref{s:3} presents the outline of the proof of Theorem \ref{t:main}: in particular, the definitions of trees, vector-valued energies as well as the statement of the energy   and tree lemmata, respectively Lemma \ref{l:energy} and \ref{l:treeest}. Section \ref{s:treeproof} contains the proof of the tree Lemma \ref{l:treeest} via the reduction to the phase-space projection Proposition \ref{p:psp}. The proof of the latter proposition is developed in Section \ref{s:pfpsp}. Section \ref{s:energyproof} contains the proof of the energy Lemma \ref{l:energy}, while  Lemma \ref{l:engbd} is proved in the concluding Section \ref{s:engbdpf}.
\subsection*{Remark} In the final stages of preparation of the present manuscript, the authors learned of the work \cite{AU2} by Amenta and Uraltsev. These authors obtain a simultaneous and independent version of Theorem \ref{t:main}, focused on the bilinear Hilbert transform in the Banach range of exponents, under the same intermediate space condition \eqref{e:inttype}. Interestingly, the methods employed in \cite{AU2} are rather different from ours: the use of phase-space projections and of the $\UMD$ Calder\'on-Zygmund estimates from \cite{DLMV1} is replaced by outer embeddings for the vector-valued wave packet transform involving telescoping (defect) energies.

The authors want to thank Alex Amenta and Gennady Uraltsev for sharing their preprint and for interesting discussions on the subject.

\section{Space-frequency model sums} \label{s:model}
\subsection{Notation}
 While our estimates are valid in any ambient space $\R^d$,   we work with $d=1$ to avoid unnecessary notational proliferation.
 However, we adopt $d$-dimensional terminology and notation whenever possible. For instance, we write $B_r(x)=\{y\in \R: |y-x|<r\}$ and simply $B_r$ in place of $B_r(0)$. Whenever possible, spatial and frequency $1$-dimensional cubes are indicated respectively by $I,\omega$. The center and sidelength of a $1$-dimensional cube $I$ are respectively denoted by $c(I),\ell(I)$. We use the Japanese bracket notation $\langle x \rangle = \sqrt{1+ |x|^2}.$

 If $m$ is a bounded function on $\R$, we denote both the corresponding $L^2(\R)$-bounded Fourier multiplier operator and its trivial extension to a Banach space $\mathcal X$ as
 \[
 T_m f(x) = \int \widehat f(\xi) m(\xi) \e^{ix\xi }\, \mathrm{d} \xi, \qquad x\in \R.
 \]
 When $\mathcal X$ is a Banach space, we keep denoting by $T$ the trivial extension $T \otimes \mathrm{Id}_\mathcal{X}$ of a linear operator $T$.

\subsection{Frequency-localized indicators} \label{ss:fli} Indicator functions, e.g. of intervals, possess perfect localization in space but poor frequency decay. We come to defining frequency-localized approximations of indicator functions by weakening such spatial localization to polynomial decay.   Let $J,K$ be  large fixed integers, $c_0$ be a small positive constant and $\eta:\R\to [0,\infty)$ have the following properties:
\[
\supp \widehat \eta \subset B_{2^{-2J}}, \qquad
c\langle x\rangle^{-2K}\leq \eta(x) \leq
\langle x \rangle^{-2K}  \quad
\forall x\in \R.
\]
We rescale $\eta$ at frequency scale $2^{jJ}$,
$
\eta_j\coloneqq 2^{jJ} \eta(2^{jJ}\cdot)
$,
and for $ E\subset \R, j \in \mathbb Z$ we introduce
\[
\chi_{E,j} = \mathbf{1}_E * \eta_j
\]
whose frequency support is contained in $B_{2^{(j-2)J}}$. This definition will be of particular interest to us when $I$ is a $J$-dyadic interval, that is, $\ell(I)=2^{-jJ}$ for some $j \in \mathbb Z$. In that case we simply write $\chi_I$ in place of $\chi_{I,j}$. Notice that, for $\chi=\chi_I$, $N=2K,\delta=2^{-2J}, c=c_0$  there holds
\begin{align}\label{e:frequencylocalized3}
& \supp \widehat \chi \subset B_{\delta \ell(I)^{-1}};
\\
\label{e:frequencylocalized1} & \chi \textrm{ real-valued},  \quad    c \left\langle \frac{x-c(I)}{\ell(I)}\right\rangle^{-N}  \leq  \chi(x) \leq C \left\langle \frac{x-c(I)}{\ell(I)}  \right\rangle^{-N},  \qquad x\in \R.
\end{align}

More generally, we denote by  $X_{I}(N,\delta,c,C)$  the class of functions satisfying conditions \eqref{e:frequencylocalized3}-\eqref{e:frequencylocalized1}.
If $\psi $ instead satisfies \eqref{e:frequencylocalized3} and
\begin{equation}
\label{e:frequencylocalized2}
| \psi(x)| \leq C \left\langle \frac{x-c(I)}{\ell(I)}  \right\rangle^{-N},  \qquad x\in \R
\end{equation}
in place of the more stringent \eqref{e:frequencylocalized1},
we say that $\psi \in \Psi_I(N,\delta, C)$. Obviously, the inclusion $X_{I}(N,\delta,c,C)\subset \Psi_I(N,\delta, C)$ holds.

 It is important to notice that if $I,I'$ are  $A$-comparable intervals, that is $I \subset AI', I'\subset A I$ and $\chi\in X_I(N,\delta,c,C)$, then $\chi \in X_{I'}(N',\delta',c',C')$ as well, for suitable values of $N',\delta',c',C'$ depending only on the comparability constant $A$ and on $N,\delta, c,C$. A similar statement applies to the classes $\Psi_I(N,\delta, C)$.

Let now $\omega $ be a $J$-dyadic \emph{frequency} interval and $\beta\geq 1$: the latter parameter will take moderate values, $1\leq \beta \ll 2^{J}$.   The class $M_{\omega}(\beta) $ will consist of those smooth functions $m $ with $\supp m \subset \beta \omega$ and adapted to $\beta \omega$ in the sense that
\[
\sup_{|\alpha| \leq C}
\sup_{\xi \in \R} \beta^{-\alpha}\ell(\omega)^\alpha  \left|\partial_\xi^\alpha m(\xi)\right| \leq 1.
\]

As customary, we will work with \emph{tiles} $t=I_t \times \omega_t\in \mathbb R \times \R$, namely the cartesian product of dyadic intervals in $\R$ of reciprocal length, to specify time-frequency localizations. In the remainder of the article, we restrict ourselves to considering $J$-dyadic intervals. Mimicking rank 1 projections in a Hilbert space, we may define classes of  multiplier operators adapted to each tile $t$ as follows. Whenever $\psi \in \Psi_{I_t}(N,\delta,C), m \in M_{\omega_t}(\beta) $, the operator
\[
S_t f(x) = \psi(x) T_{m} f(x) = \int_{\R} \psi(x) m(\xi) \widehat f(\xi) \e^{i x \cdot \xi}\, \mathrm{d} \xi
\]
is said to belong to the class $ {\mathbb S}_t(N,\delta, \beta,c,C)$ of $t$-localized operators.
From the rapid decay of the kernel of $T_m$ and the adaptedness of $\psi$ it follows that
\begin{equation}
\label{e:singlescaleM}
\left|S_t f(x)\right|_{\mathcal X} \lesssim \left\langle \frac{\mathrm{dist}(x,I_t)}{\ell(I_t)} \right\rangle^{-100}M(|f|_{\mathcal X}) (x), \qquad x\in \R.
\end{equation}
Here  $\mathcal X$  may be any Banach space, not necessarily $\UMD$.
\begin{rem}
In the remainder of the paper, the values $N,\delta,\beta,c,C$ will vary within a fixed range \[c\geq \bar{c}, \quad C\leq \bar C, \quad N\geq \bar{N}, \quad2^{-2J}\leq \delta \leq 2^{-(1.5)J} , \quad 1\leq \beta\leq \bar{\beta}.\] For simplicity we keep these parameters implicit and  omit them from the notations for $X_{I}(N,\delta,c,C)$, $\Psi_I(N,\delta, C)$, $M_{\omega}(\beta)$ and  $ {\mathbb S}_t(N,\delta, \beta,c,C)$. Therefore, the reader is warned that the precise values of these parameters may vary from line to line without explicit mention.

An advantageous example of usage for this convention is that whenever $\chi\in X_{I}$, the functions $\chi^m \in  X_{I}$ as well for small values of $m \in \mathbb N$.
\end{rem}
In our arguments we will make use of a form of Bernstein's inequality involving approximate indicators, in particular functions of the classes ${X_I}$ described above. This  is a known phenomenon in the literature, see e.g.\ \cite[Lemma 5.4]{MTT2}; we give the proof as we are in the vector-valued context.
\begin{lem} \label{l:bernstein}
Let $R>0$,  $\mathcal X$ be a Banach space and $f$ be an $\mathcal X$-valued function on $\R$ with \[\mathrm{supp}\, \widehat f \subset B_{R}.\]
Let  $w:\R \to (0,\infty)$  be essentially constant at scale $R^{-1}$, namely
\[
A^{-1} \left\langle R {|x-y|}  \right\rangle^{-100} \leq
\frac{w(x)}{w(y)} \leq A \left\langle R {|x-y|}  \right\rangle^{100} , \qquad x,y\in \R
\]
for some positive constant $A$. Then for all $0<\alpha\leq 1 $
\[
\|wf\|_{L^\infty(\R; \mathcal X)} \lesssim_{A,\alpha} R^\alpha\|wf\|_{L^{\frac{1}{\alpha}}(\R; \mathcal X)}.
\]
\end{lem}
\begin{proof}Let $\phi$ be a smooth nonnegative function with $\phi=1$ on $B_{R}$ and $\phi=0$ on $B_{2R}$. Notice that
$
|\widehat{\phi}(x)| \lesssim  R \langle R|x|\rangle^{-200},
$ for all $x\in \R$.
  Then $f=f*\widehat{\phi}$, and
\[
\begin{split}
|w(x) f(x)|_{\mathcal X} &\lesssim_{\alpha,A} R \int \frac{ |w(y) f(y)|_{\mathcal X}}{\left\langle R {|x-y|}  \right\rangle^{100}} \mathrm{d} y \\ &\leq R^\alpha \|wf\|_{L^{\frac{1}{\alpha}}(\R; \mathcal X)} \left(\int\frac{  \, R \mathrm{d} y}{\left\langle R {|x-y|}  \right\rangle^{100}}   \right)^{1-\alpha} \lesssim R^\alpha \|wf\|_{L^{\frac{1}{\alpha}}(\R; \mathcal X)}
\end{split} \] as claimed. The proof is complete.
\end{proof}
We will apply the lemma above to $w=\chi \in {X_I}$ for values  $R\sim \ell(I)$.
\subsection{Rank 1  forms}
Trilinear multiplier forms of the type \eqref{e:Lambdam} admit a discretization in term of \emph{tri-tiles}. We say that the ordered triple  of tiles
 $P=(P_1,P_2,P_3)$ is a \emph{tri-tile} if
 \[
 I_{P_1}=I_{P_2}=I_{P_3}=:I_{P}.
 \]
Tri-tiles specify the space-frequency essential support of single scale multiplier forms, spatially concentrated on $I_P$ and frequency supported on the \emph{frequency cube}
\begin{equation}
\label{e:q}Q_P=\omega_{P_1}\times \omega_{P_2}\times \omega_{P_3} .
\end{equation} As condition \eqref{decay} is invariant under the one-parameter family of translations along the vector ${\beta}^\perp$, the decomposition of the forms \eqref{e:Lambdam} is realized with a one-parameter, or \emph{rank 1}, family of tri-tiles, which is defined as follows.

Let $J>10$ be a (large) constant depending on $\Delta_{\mathbf \beta}$.
The collection of tri-tiles $\mathbb{P}$ is of rank 1 if the following hold.
\begin{itemize}
\item[a.] The collections $ \{I_P:P \in \Pp\}$ and  $\{ \omega_{P_k}: P \in \Pp \}$, $k=1,2,3$ are  $O(2^J)$  scale-separated $J$-dyadic grids;
\item[b.] If $P\neq P' \in \Pp$ are such that  $I_P=I_{P'}$ then  $\omega_{P_j}\cap\omega_{P'_j}=\varnothing$ for each  $j \in\{1,2,3\}$;

\item[c.] Denote by  $ {\omega_{P}}$  the convex hull of the intervals $3\omega_{P_k},$ $k=1,2,3.$ If $P,P'\in \Pp$ are such that $\omega_{P_k} \subset  \omega_{P'_k}$ for some $k\in \{1,2,3\}$ then $2^{\frac{J}{2}} { \omega_{P}} \subset 2^{\frac{J}{2}}   {\omega_{P'}}$;

\item[d.] if $P,P'\in \Pp$ are such that $\omega_{P_j} \subset  \omega_{P'_j}$ for some $j\in \{1,2,3\}$ then $3\omega_{P_k} \cap  3 \omega_{P'_k}=\varnothing$ for $k\in \{1,2,3\}\setminus\{j\}$.
\end{itemize}
The singular multiplier forms $\Lambda_m$ then  lie in the convex hull of the  \emph{tri-tile forms}
\begin{equation}
\label{e:tri-tile}
\Lambda_{\mathbb P}(f_1,f_2,f_3) = \sum_{P \in \mathbb P }\int_{\R} \prod_{k=1}^3 \chi_{I_P}(x) T_{m_{P_k}} f_k (x) \, \mathrm d x
\end{equation}
where $\mathbb P$ is  a finite subset of a rank 1 collection of tri-tiles, $\chi_{I_P}\in X_{I_P}$ has been defined in Subsection \ref{ss:fli} and $m_{P_k} \in M_{\omega_{P_k}}(1), P \in \mathbb P$ satisfies the consistency condition\begin{equation}
\label{e:conscond}
Q_{P}= Q_{P'} \implies m_{P_k}= m_{P'_k}, \; k=1,2,3,
\end{equation}
 referring to \eqref{e:q}.
Therefore,    to bound $\Lambda_m$ it suffices to bound  $\Lambda_{\mathbb P}$ uniformly.
A detailed proof of these statements, in a more general context, is given in \cite[Section 5]{MTT}, see also \cite{MTT2,ThWPA}. The remainder of the article will be devoted to the proof of such uniform bounds for $\Lambda_{\mathbb P}$. Note that the operators
\[
f\mapsto \chi_{I_P} (T_{m_{P_k}} f)
\]
belong to the class $\mathbb S_{P_k}$, for $k=1,2,3$.
\section{Proof of Theorem \ref{t:main}: tree and energy estimates}
\label{s:3}
In this section, after devising the necessary definitions in our context, we present the statements of the three main lemmas, which may then be combined to prove Theorem \ref{t:main} in a standard fashion.

 We first introduce \emph{trees}, roughly speaking, collections of tri-tiles sitting at a common frequency and spatially localized to an interval. Then we define   \emph{tree operators}, that is modulated Calder\'on-Zygmund localized operators associated to each tree. These are  used to define the \emph{energy} of a certain $\mathcal X$-valued function with respect to a set of tri-tiles $\mathbb P$: this is a sort of localized maximal $L^q(\R; \mathcal X)$-norm of tree operators coming from $\mathbb P$.

 Finally, we state the main steps in the proof of Theorem \ref{t:main}. The first is the \emph{energy lemma}, which allows us to decompose any given collection of tri-tiles into   are unions of trees of controlled energy for each function $f_k$ and bounded spatial support. The second is the \emph{tree lemma}, which provides a bound of the tri-tile form \eqref{e:tri-tile} when $\mathbb P$ is a tree. The proof of this lemma is one of the main novelties of this article, as it relies on a combination of the multilinear $\UMD$ CZ theory of \cite{DLMV1} with newly developed phase-space projections adapted to the vector-valued setting.

 We remark here that a standard combination of Lemmata \ref{l:treeest},
\ref{l:energy} and \ref{l:engbd} yields a range of restricted weak type estimates for the forms \eqref{e:tri-tile}: the elementary procedure is identical to that leading to \cite[Corollary 9.2]{HLP}.  These estimates then entail Theorem \ref{t:main} by standard multilinear restricted weak type interpolation, see e.g.\ \cite{ThWPA}. This deduction is the same as that of   \cite[Theorem 9.3]{HLP} from \cite[Corollary 9.2]{HLP}. We omit the details.

\subsection{Trees} \label{ss:trees}
Rank 0 subcollections of tri-tiles $\mathbb P$, whose associated forms $\Lambda_{\mathbb P}$ are discretized multilinear CZ type
multipliers, are called \emph{trees}. We work with a specific notion of tree which satisfies certain additional properties along the lines of  \cite[Section 4]{MTT2}.
We say that $\T\subset \Pp$ is a tree having $(I_\T,\xi_\T)$ as top data if the following conditions hold.

\begin{itemize}\item[a.] $I_\T$ is a $J$-dyadic interval, $\xi_\T\in \R$  and \[
I_P \subset I_\T, \quad \xi_\T \in \omega_P \qquad \forall P \in \T.\]
\item[b.] The \emph{frequency localization sets} $\mathbf{Q}_\T=\{Q_P: P \in \T\}$ is such that
\[
Q,Q'\in \mathbf{Q}_\T, \, \ell(Q)=\ell(Q') \implies Q=Q';
\]
namely, there is only one frequency localization for each  $J$-dyadic scale.
\item[c.] The \emph{spatial localization sets}
\[
E_{Q,\T}= \bigcup \{I_P: P \in \T: Q_P = Q\}, \qquad Q \in \mathbf Q_\T
\]
are nested, that is
\[
Q,Q'\in \mathbf Q_\T, \,\ell(Q)\leq \ell(Q') \implies E_{Q,\T} \supset E_{Q',\T}.
\]
\end{itemize}{
It is convenient to denote by $\mathbf{j}_\T=\{j\in \mathbb Z: \ell(Q) = 2^{j J} $ for some $Q \in \mathbf{Q}_\T\}$, the frequency scales appearing in $\T$. Then
\begin{equation}
\label{e:Tj} \T= \bigcup_{j \in \mathbf{j}_\T} \T(j), \qquad \T(j)=\{P\in \T : \ell(Q_P)=2^{jJ}\}.
\end{equation}
  We also take the opportunity here to observe that trees constructed via greedy selection processes, such as the one in the proof of Lemma \ref{l:energy} below, satisfy properties b.\ and c.\ automatically. This is proved in e.g.\ \cite[Lemmata 4.4 and 4.7]{MTT2}. Furthermore, the smoothing properties of \cite[Lemmata  4.10, 4.11, 4.12]{MTT2} also hold. We  will only make use of these properties within the proof of the phase-space projection estimates: we will then recall what is needed, and send to \cite{MTT2} for detailed statements.

\subsection{Tree operators}
We say that the tree $\T$ is $k$-lacunary, $k=1,2,3$ if
\begin{equation}
\label{e:klac}
 \{3\omega_{P_k}:P \in \T_A\} \textrm{ are a pairwise disjoint collection for } k \in A .
\end{equation}
A consequence of property d.\ of  rank 1 collections is that  each tree $\T$ can be written as the disjoint union
\begin{equation}\label{Treesplit}
\T=\bigcup_{\substack {A\subset \{1,2,3\} \\ \#A \geq 2}} \T_{A}
\end{equation}
where each $\T_A$ is a tree with the same top data as $\T$ and has the additional property \eqref{e:klac} for $k \in A$.

We now introduce  tree operators associated to $k$-lacunary trees. For our purposes here, we need a more nuanced notion than the usual, e.g.\ appearing in \cite{CDPO,HLC,LT1,LT2}, fully discretized tree operator
\begin{equation}
\label{e:treeopold}
f\mapsto \sum_{P \in \mathsf{T}  } \langle f, \varphi_{P_k}\rangle \varphi_{P_k}, \qquad k \in A
\end{equation}
where $\varphi_{P_k}$ is a wave packet adapted to the tile $P_k$. Let $\T$ be a $k$-lacunary tree.  A (scalar) tree operator of $k$-th type is the linear operator
\begin{equation}
\label{e:treeop}
T_\T  f=
\sum_{P \in \mathsf{T} } S_{P_k} f
\end{equation}
where each  $S_{P_k}\in \mathbb S_{P_k}$, $P\in \mathsf T$. When $\xi_\T=0$, the  defined tree operator   is a pseudo-differential operator with symbol
\[
a  (x,\xi) =
\sum_{P \in \mathsf{T}}  \psi_{P_k}(x) m_{P_k} (\xi)
\]
where each $\psi_{P_k}\in \Psi_{I_P}$ and $m_{P_k}\in M_{\omega_{P_k}}$. A routine computation verifies that this symbol is uniformly of class $S^0_{1,1}$, see e.g.\ \cite[p.\ 288]{MTT2}. Further, as the intervals $\{ \omega_{P_k}: P\in  \mathsf{T}\} $ are pairwise disjoint, $T_\T$ is uniformly $L^2(\R)$ bounded. Relying on these two observations, we gather that $T_\T$ is an $L^2(\R)$-bounded Calder\'on-Zygmund operator, see for instance the discussion at \cite[p.\ 271]{SteinB}.  Therefore, $T_\T$  satisfies uniform $L^q(\R; \mathcal{X})$ bounds, $1<q<\infty$, as well   $L^\infty(\R; \mathcal{X})\to \mathrm{BMO}(\R; \mathcal{X})$ estimates, whenever $\mathcal X$ is a $\UMD$ Banach space. In fact, by modulation invariance, we may remove the $\xi_\T=0$ assumption and
conclude that tree operators  $T_\T$  are uniformy $L^q(\R; \mathcal{X})$ bounded, when $1<q<\infty$.
\subsection{Energy and energy lemma}

This definition is instead a re-elaboration of  \cite[Section 8]{HLC}. If $q\geq 2$, $\mathsf T$ is a $k$-lacunary tree and $f$ is a $\mathcal X$-valued function, we define
\begin{equation}
\label{e:size} \|f\|_{\mathsf{T},k,q} = \sup \frac{1}{|I_{\mathsf{T}
}|^{\frac1q}} \left\|  T_{\mathsf T} f \right\|_{L^q(\R; \mathcal{X})}
\end{equation}
where the sup is taken over all possible choices of type $k$  tree operators  $T_{\mathsf T}$.
 We also find convenient to define a maximal version: for each set $\mathbb P$ of tri-tiles,
\begin{equation}
\label{e:engdef}
\mathsf{eng}_k(f)(\mathbb{P};q) = \sup_{\substack{\mathsf T \subset \mathbb P \\ \T \, k-\textrm{lacunary }}} \|f\|_{\mathsf{T},k,q}.
\end{equation}
The next lemma is a variation of e.g.\ \cite[Corollary 9.6]{HLC}.
\begin{lem} \label{l:engbd} Let $\mathbb P$ be a finite collection of tri-tiles. Then
\[
\mathsf{eng}_k(f)(\mathbb{P}; q) \lesssim \sup_{P \in \mathbb P} \inf_{I_P} \mathrm{M}(|f|_\mathcal{X}) .
\]
\end{lem}
Although the arguments of \cite{HLC} may be adapted to the context of Lemma \ref{l:engbd}, we provide a more direct proof in Section \ref{s:engbdpf}.

In the last main lemma, the quantitative assumption of $\mathcal X_k$ being an interpolation space is used. We could alternatively bring forth definitions akin to the tile-type of a Banach space in \cite{HLC,HLP,HLPC}, which is a formal consequence of our intermediate $\UMD$ assumption, but for simplicity and lack of examples we give up on this additional formal generality.
\begin{lem} \label{l:energy} Suppose $\mathcal X$ is   $q_\mathcal{X}$-intermediate and let $f \in L^\infty(\R; \mathcal X)$ be subordinated to the finite measure set $F$, namely $|f|_\mathcal{X} \leq \mathbf{1}_F$. Fix $q> q_\mathcal{X}$ and let $\mathbb P$ be a finite set of tri-tiles. Then
$
\mathbb P = \mathbb P^+ \cup \mathbb P^-
$
with the property that
\[
\mathsf{eng}_k(f)(\mathbb{P}^-; q_\mathcal{X}) \leq 2^{-1} \mathsf{eng}_k(f)(\mathbb{P};q_\mathcal{X})
\]
and that $\mathbb P^+$ is a union of  trees $\mathsf{T} \in \mathcal T$ with the property that
\[
\sum_{\mathsf{T} \in \mathcal T} |I_{\mathsf{T}}|\lesssim_q \left[\mathsf{eng}_k(f)(\mathbb{P}; q_\mathcal{X})\right]^{-q} |F|.
\]
\end{lem}
The proof of Lemma \ref{l:energy} is a revisitation of the steps leading to \cite[Proposition 8.4]{HLC} and is postponed to Section \ref{s:energyproof}. Note that Lemma \ref{l:energy} is the only   main step of the proof of Theorem \ref{t:main} where a $q_{\mathcal X}$-intermediate assumption is used.

The final main tool of the proof of Theorem \ref{t:main} is  a bound on the forms \eqref{e:tri-tile} in terms of energy parameters in the particular case where the collection $\mathbb P$ is a tree.
\begin{lem}  \label{l:treeest} Let  ${\mathcal{X}_k}$, $k=1,2,3$  be $\UMD$ spaces and \begin{equation}
\label{e:typetree} 2\leq q_1,q_2, q_3 < \infty, \qquad
\sum_{k=1}^3 \frac{1}{q_{k}} \geq 1.
\end{equation}
 Let $\mathsf T$ be a tree. With reference to \eqref{e:tri-tile} there holds
\[
{|\Lambda_{\mathsf T }(f_1, f_2, f_3)|}  \lesssim |I_{\mathsf T}| \prod_{k=1}^3\mathsf{eng}_k(f_k)(\mathsf T; q_k).
\]
\end{lem}
The proof of Lemma \ref{l:treeest} uses a novel vector-valued version of the phase-space projection technique of \cite{MTT2,DT} in conjunction with \cite[Theorem 1.2]{DLMV1} and is given in  Section \ref{s:treeproof}.

\section{Phase space projections and the proof of the tree lemma} \label{s:treeproof}
We develop phase-space projections in the vector-valued context and  combine them with the bounds for vector-valued extensions of bilinear CZ kernels to prove Lemma \ref{l:treeest}. The following treatment is an adaptation of the construction made in \cite[Sec.\ 7 and 8]{MTT2}. Our arguments are more involved due to the  vector-valued nature of the involved functions. However, we take advantage of a significant simplification in that no uniformity issues are considered: in the language of \cite{MTT2} the indices $\mathbf{m}_i$ are all zero.  Uniform estimates in the vector-valued context will be the object of future work.

In the main proposition of this section, we  make use of Littlewood-Paley projections as follows.   The operator $T_j$ is meant to be a Fourier multiplier whose symbol is real, even, supported on $[-2^{J(j+2)}, 2^{J(j+2)}]$ and equals one on $[-2^{J(j+1)}, 2^{J(j+1)}]$ and $S_j=T_{j}-T_{j-1}$. The projections $S_j$ appear also in Lemma \ref{l:cz} below.
\begin{prop}[phase-space projections] \label{p:psp} For $k=1,2,3,$ let ${\mathcal{X}_k} $ be  a $\UMD$ space and  $q_k\in [2,\infty)$. Let  $\mathsf{T}$ be a tree with the following properties:
\begin{itemize}
\item[i.]
 $\xi_\T=0$;
 \item[ii.]
$\T$ is lacunary in the components $k\in A$ and overlapping in the components $k \in B$, with $A\cup B=\{1,2,3\}$ disjoint union and $\#A\in \{2,3\}$;
\item[iii.] the separation of scales condition \[\inf\{|j-j'|: j,j'\in \mathbf{j}_\T, j\neq j'\} \geq M\] holds for some large constant $M$.
\end{itemize}
Choose $\{S_{P_k}\in \mathbb S_{P_k}, P \in \T, k \in \{1,2,3\}\}. $
 Then there are   linear operators $\Pi_k$ with the following properties.
\begin{itemize}
\item[a.] If $p\geq q_{k}$, there holds
\begin{equation}
\label{e:pspa}
\|\Pi_k f\|_{L^{p}(\R; {\mathcal{X}_k})} \lesssim |I_{\mathsf T}|^{\frac{1}{p}}  \mathsf{eng}_k(f)(\mathsf T; q_k);
\end{equation}
\item[b.] if $k\in A$, for all $j\in \mathbf{j}_\T$, referring to \eqref{e:Tj}
\[
\sum_{P \in \T(j)} S_{P_k} f=    S_j (\Pi_k f).
\]
\item [c.] If $p\geq q_{k}$, $k\in B$, $j_0\in \mathbf{j}_{\mathsf{T}}$ and $\ell(I_0) =2^{-Jj_0}$,
\begin{equation}
\label{e:pspc}
\left\| \mathbf{1}_{I_0}  \sum_{P \in \T(j_0)} S_{P_k} (f - \Pi_k f)\right\|_{L^{p}(\R; {\mathcal{X}_k})} \lesssim  {|I_0|^{\frac{1}{p}}}\mathsf{eng}_k(f)(\mathsf T; q_k) \int_{\R} \chi_{I_0}(x) \mu_{j_0}(x) \frac{\mathrm{d} x}{|I_0|}.
\end{equation}
where  $\mu_j$ is as in \cite[eq.\ (51)]{MTT2}. Furthermore
\begin{equation}
\label{e:pspc2}
 \left\| \mathbf{1}_{I_0} \sum_{P \in \T(j_0)} S_{P_k}  \Pi_k f\right\|_{L^{p}(\R; {\mathcal{X}_k})}  \lesssim |I_{0}|^{\frac{1}{p}}  \mathsf{eng}_k(f)(\mathsf T; q_k).
\end{equation}
\end{itemize}While the operators $\Pi_k$ depend on the choice of $\{S_{P_k}\in \mathbb S_{P_k}, P \in \T, k \in \{1,2,3\}\}, $ the  estimates above are uniform over such choice.
\end{prop}
The proof of Proposition \ref{p:psp} is postponed to  the next section. Herein, we proceed to showing how such proposition may be coupled with the main result of \cite{DLMV1} to obtain the tree estimate we claimed in Lemma \ref{l:treeest}. The next subsection contains some preliminaries while the main line of argument is given in Subsection~\ref{ss:pflemmatree}
\subsection{Preliminaries} We begin with a preliminary localized single scale estimate for tree operators which will be of use  towards Lemma \ref{l:treeest} as well as in Section \ref{s:pfpsp}.
\begin{lem} \label{l:sscale} Let $I_0$ be a $J$-dyadic interval with $\ell(I_0)=2^{-j_0J}$, $\psi_{I_0}\in \Psi_{I_0}$, $S_{P_k}\in \mathbb{S}_{P_k}$ for each $P \in \T(j_0)$. Then
\begin{equation}
\label{e:pspc2a}
 \left\|\psi_{I_0}  \sum_{P \in \T(j_0)} S_{P_k}    f\right\|_{L^{p}(\R; {\mathcal{X}_k})}  \lesssim |I _0|^{\frac{1}{p}}  \mathsf{eng}_k(f)(\mathsf T;q_k), \qquad    q_k\leq p\leq \infty.
\end{equation}
\end{lem}
Arguing by interpolation, it suffices to prove the extremal cases.
\begin{proof}[Proof of \eqref{e:pspc2a} for $p=q_k$] Let  $n\in \mathbb N$. There are at most two $P\in \T(j_0)$ such that  $\dist(I_P,I_0) =n2^{-j_0J}$. Fix such a $P$. It then suffices to estimate
\[
 \left\|\psi_{I_0}   S_{P_k}    f\right\|_{L^{q_k}(\R; {\mathcal{X}_k})}  \lesssim \langle n
\rangle^{-100} |I _0|^{\frac{1}{q_k}}  \mathsf{eng}_k(f)(\mathsf T;q_k).
\] Write $S_{P_k} f= \psi_{P} T_{m_{P_k}} f $. Then $\widetilde\psi\coloneqq \langle n
\rangle^{100}  \psi_{I_0}\psi_{P} \in \Psi_{I_P} $
and the estimate in the last display simply follows from the definition of $\mathsf{eng}_k(f)(\T; q_k)$.
\end{proof}
\begin{proof}[Proof of \eqref{e:pspc2a} for $p=\infty$] The function we are estimating has frequency support in a ball of radius $O(2^{j_0J})$. Then this case follows from the case $p=q_k$  and a straightforward application of Lemma \ref{l:bernstein}.
\end{proof}
We then particularize the definition \eqref{e:tri-tile} to the case where $\mathbb P$ is our tree $\mathsf \T$. By modulation, translation and scaling invariance, we may reduce  Lemma \ref{l:treeest} to the case $\xi_\T=0, I_\T=[0,1)$.   Notice that, referring to \eqref{e:q}, \eqref{e:Tj}, we have  $Q_{P}=Q_{P'}\coloneqq Q_j$ for all $P, P'\in \T(j)$,  and that consequently $m_{P_k}=m_{P'_k}\coloneqq m_{j,k}$ for all $P,P' \in \T(j),k=1,2,3$. Therefore we may set for $j \in\mathbf{j}_\T$
\begin{align}
\label{e:treepre1} &\tilde \chi_j \coloneqq \sum_{P \in \T(j)} \chi_{I_P}^3
\\
\label{e:treepre2}  & \tilde\pi_{j,k}\coloneqq  T_{m_{j,k}}, \qquad k=1,2,3
\end{align}
and rewrite, and subsequently estimate, \eqref{e:tri-tile} for $\T=\mathbb P$ as, cf.\ \cite[eq.\ (44)]{MTT2}
\begin{equation}
\label{e:treepre3}
\left|\Lambda_\T(f_1,f_2,f_3)\right| = \left| \sum_{j \in \mathbf{j}_\T} \int_{\R} \tilde \chi_j \prod_{k=1}^3 \tilde \pi_{j,k} f_k  \right|
\leq  \sum_{j \in \mathbf{j}_\T} \left|   \int_{\R} \tilde \chi_j \prod_{k=1}^3 \tilde \pi_{j,k} f_k  \right|.
\end{equation}
An analogous argument to the one that proves \cite[Lemma 7.3]{MTT2}, provided that  the vector-valued Lemma \ref{l:sscale} is used in place of \cite[Lemma 7.2]{MTT2}, yields that
\begin{equation}
\label{e:treepre4}
  \sum_{j \in \mathbf{j}_\T} \left|   \int_{\R} \tilde \chi_j(x) \prod_{k=1}^3 \tilde \pi_{j,k} f_k (x) -  \prod_{k=1}^3 \tilde \chi_j(x) \tilde \pi_{j,k} f_k(x) \, \mathrm{d} x   \right| \lesssim \prod_{k=1}^3\mathsf{eng}_k(f_k)(\mathsf T; q_k)
\end{equation}
which is acceptable for the estimate of Lemma \ref{l:treeest}. Therefore we have reduced the Lemma to proving that
\begin{equation}
\label{e:treepre5} \left|
  \sum_{j \in \mathbf{j}_\T} \eps_j   \int_{\R}   \prod_{k=1}^3 \tilde \chi_j  \tilde \pi_{j,k} f_k \right|     \lesssim \prod_{k=1}^3\mathsf{eng}_k(f_k)(\mathsf T; q_k)
\end{equation}
uniformly over choices of unimodular coefficients $\{\eps_j: j \in \mathbf{j}_\T\} $
which is the core of the argument, and is left for the next subsection.

The final preliminary result is a H\"older type estimate for the vector-valued extension of a classical trilinear paraproduct form. Such estimate is a particular case of the main result of \cite{DLMV1} and depends only on the $\UMD$ property of the spaces involved.
\begin{lem} \label{l:cz} Let $\{p_k:k=1,2,3\}$ be a H\"older tuple of exponents with $1<p_k<\infty$ for all $k=1,2,3$. Let $\mathcal X_k$ be $\UMD$ spaces with a trilinear contraction $\prod_{k=1}^3 \mathcal X_k \to \mathbb C$. Let $g_k\in (L^1(\R) \cap L^\infty (\R)) \otimes \mathcal X_k$, for $k=1,2,3$. Then
\[
\left| \int
\sum_{j \in \mathbf{j}_\T} \eps_j (\tilde \pi_{j,1}g_1) (S_{j}g_2)   (S_{j}g_3) \right| \lesssim \prod_{k=1}^3 \|g_k\|_{L^{p_k}(\R; \mathcal X_k)}.
\]
\end{lem}
\begin{proof} Recall that  $m_{j,1}$, the symbol of $\tilde\pi_{j,1}$, is  adapted and supported in a moderate dilate of $Q_{j,1}$, which is a dyadic interval of length $2^{jJ}$ and such that $2^{\frac J2 }Q_{j,1}$ contains the origin. Thus $m_{j,1}$ vanishes outside $|\xi|\leq 2^{J(j+\frac12)}$.  The  symbol $\Psi_j$  of $S_j$  is supported on $2^{J(j+1)}\leq |\xi|\leq 2^{J(j+2)}$.  Let $g_k\in L^1(\R) \cap L^\infty (\R)$ be scalar functions. Then Plancherel's equality yields
\begin{equation}
\label{e:cz1}
\sum_{j \in \mathbf{j}_\T} \eps_j (\tilde\pi_{j,1}g_1) (S_{j}g_2)   (S_{j}g_3) = \langle O(g_1,g_2), \overline{g_3} \rangle,
\end{equation}
where $O$ is the bilinear Fourier multiplier operator
\[\begin{split}
&O(g_1,g_2)(x) = \int_{\R\times \R} \widehat{g_1} (\xi)\widehat{g_2} (\xi) m(\xi_1,\xi_2) e^{2\pi i x(\xi_1+\xi_2)}\mathrm{d} \xi_1 \mathrm{d} \xi_2, \qquad x \in \R,
\\
&m(\xi_1,\xi_2) \coloneqq\sum_{j \in \mathbf{j}_\T} \eps_j m_{j,1}(\xi_1)\Psi_j(\xi_2)\Psi_j(-\xi_1-\xi_2).
\end{split}
\]
The support and smoothness conditions on $ m_{j,1},\Psi_j$ imply that $m$ satisfies the  Coifman-Meyer condition  multiplier and thus $O$ is  a bilinear CZ kernel operator. We may then use \cite[Theorem 1.2]{DLMV1} to conclude that $O$ extends to a bounded bilinear operator
\[
L^{p_1}(\R; \mathcal X_1) \times
L^{p_2}(\R; \mathcal X_2) \to L^{p_3'}(\R; \mathcal X_3').
\]
As \eqref{e:cz1} continues to hold for $g_k\in (L^1(\R) \cap L^\infty (\R)) \otimes \mathcal X_k$, the vector-valued bound of the above display and duality complete the proof of the lemma.
\end{proof}
\subsection{Proof of Lemma \ref{l:treeest}, estimate \eqref{e:treepre5}} \label{ss:pflemmatree}
By the condition \eqref{e:typetree}, we may find a H\"older tuple $p_1,p_2,p_3$ with $q_{k}\leq p_k<\infty$.
 To apply Proposition \ref{p:psp} it is useful to keep in mind the equalities
\begin{equation}
\label{e:treepf}
   \int_{\R}   \prod_{k=1}^3 \tilde \chi_j  \tilde \pi_{j,k} f_k   = \int_\R \prod_{k=1}^3 \left( \sum_{P \in \T(j)} S_{P_k} f_k \right), \qquad j \in \mathbf{j}_\T
\end{equation}
having called $S_{P_k}\in \mathbb S_{P_k}$ the operator $f\mapsto \chi_{I_P} T_{m_{j,k}} f_k$.
We first handle the easy case where $B=\varnothing.$ Applying part b.\ of Proposition \ref{p:psp} to each $f_k$, for each $j \in \mathbf{j}_\T$, we have
\[
\sum_{j \in \mathbf{j}_\T} \eps_j   \int_{\R}   \prod_{k=1}^3 \tilde \chi_j  \tilde \pi_{j,k} f_k    =
\sum_{j \in \mathbf{j}_\T} \eps_j   \int_{\R}   \prod_{k=1}^3 S_j (\Pi_k f_k).
\]
For $k=1,2,3$ let $\sigma_k=\{\sigma_{j,k}:j\in \mathbf{j}_\T\}$ be a sequence of i.i.d.\ random variables which take the values $1, -1$ with equal probability. We denote the expectation with respect to $\sigma_k$ by $\mathbb E^k$. Using \cite[Lemma 4.1]{DLMV1}, H\"older's inequality, $L^{p_k}$-bounds for the $\mathcal X_k$-valued randomized square function (this holds since $\mathcal X_k$ is $\UMD$), and subsequently part a.\
of Proposition \ref{p:psp}, there holds
\[
\begin{split}
&\quad \left|\sum_{j \in \mathbf{j}_\T}  \eps_j\int_{\R}   \prod_{k=1}^3  S_j (\Pi_k f_k)  \right|\\
& \lesssim
\prod_{k=1}^3 \left( \mathbb{E}^k  \int_{\R} \left| \sum_{j \in \mathbf{j}_\T} \sigma_{j,k}   S_j (\Pi_k f_k) \right|_{\mathcal X_k}^{p_k} \right)^{\frac{1}{p_k}}
\lesssim \prod_{k=1}^3  \|\Pi_k f_k)\|_{L^{p_k}(\R; \mathcal X_k)} \lesssim   \prod_{k=1}^3\mathsf{eng}_k(f_k)(\mathsf T; q_k),
\end{split}
\]
which is the claim \eqref{e:treepre5}.

We turn to the harder case where $\#A=2$. By symmetry we may work with $B=\{1\} $. We use Proposition \ref{p:psp} to bound the left hand side of \eqref{e:treepre5} by $\mathsf{MAIN}+\mathsf{ERR}_1 + \mathsf{ERR}_2$ where
\begin{align}
\label{e:treepf1}
&\mathsf{MAIN}\coloneqq\left|
  \sum_{j \in \mathbf{j}_\T} \eps_j   \int_{\R}    \tilde \pi_{j,1}(\Pi_1 f_1)  \prod_{k=2}^3 \tilde \chi_j  \tilde \pi_{j,k} f_k \right|
  \\
  & \mathsf{ERR}_1\coloneqq
  \sum_{j \in \mathbf{j}_\T} \sum_{\ell(I)=2^{-jJ}}   \int_{I} |\tilde\chi_j  \tilde \pi_{j,1}(f_1-\Pi_1 f_1) |_{\mathcal{X}_1} \prod_{k=2}^3\left| \tilde \chi_j  \tilde \pi_{j,k} f_k \right|_{\mathcal{X}_k},
  \\
  & \mathsf{ERR}_2\coloneqq
  \sum_{j \in \mathbf{j}_\T} \sum_{\ell(I)=2^{-jJ}}   \int_{I} \zeta_j  |\tilde \pi_{j,1}(\Pi_1 f_1) |_{\mathcal{X}_1} \prod_{k=2}^3\left| (\tilde \chi_j)^{\frac16}  \tilde \pi_{j,k} f_k \right|_{\mathcal{X}_k}, \quad \zeta_j\coloneqq|(\tilde\chi_j^2-\tilde\chi_j^3)|(\tilde\chi_j)^{\frac12}
\label{e:treepf2}
\end{align}
the second  and third of which are error terms. In $\mathsf{ERR}_j$ the sum over $I$ is over all $J$-dyadic intervals of a fixed length $2^{-jJ}$.

We first handle the error terms: via H\"older's inequality with exponents $p_k$, and a combination of \eqref{e:pspc} for the $\mathcal{X}_1$ with Lemma \ref{l:sscale} for the $\mathcal X_k$ factors, $k=2,3$, we achieve the estimates
\[
 \mathsf{ERR}_1\lesssim \left( \prod_{k=1}^3\mathsf{eng}_k(f_k)(\mathsf T; q_k) \right)  \sum_{j \in \mathbf{j}_\T} \sum_{\ell(I)=2^{-jJ}} \int \chi_{I} \mu_j \lesssim \left(\prod_{k=1}^3\mathsf{eng}_k(f_k)(\mathsf T;q_k) \right) \sum_{j \in \mathbf{j}_\T}   \int \mu_j.
\]
As detailed in \cite[p.286, Lemma 4.12]{MTT2} the $j$-summation above is $\lesssim |I_\T|=1$, showing that $\mathsf{ERR}_1$ complies with the right hand side of \eqref{e:treepre5}. The second error term is bounded also  using H\"older, followed by the single scale estimates \eqref{e:pspc2} and Lemma \ref{l:sscale}:
\[ \mathsf{ERR}_2\lesssim \left( \prod_{k=1}^3\mathsf{eng}_k(f_k)(\mathsf T;q_k) \right)  \sum_{j \in \mathbf{j}_\T} \sum_{\ell(I)=2^{-jJ}} |I| \|\mathbf{1}_I\zeta_j\|_\infty.
\]
The $j,I$ summation above is also bounded by $|I_\T|$ via \cite[Lemma 4.8]{MTT2}. We omit the details.

We move to the main term. Using part b.\ of the Proposition, we recognize that
\[
\mathsf{MAIN} = \left| \sum_{j \in \mathbf{j}_\T} \eps_j \int (\tilde\pi_{j,1}\Pi_1 f_1) (S_{j}\Pi_2)   (S_{j}\Pi_3)    \right| \lesssim \prod_{j=1}^3 \|\Pi_k f_k\|_{L^{p_k}(\R; \mathcal X_k)} \lesssim\prod_{k=1}^3\mathsf{eng}_k(f_k)(\mathsf T;q_k)
\]
having used Lemma \ref{l:cz} for the first bound, and \eqref{e:pspa} for the second. This completes the proof of Lemma \ref{l:treeest}.
\section{Proof of Proposition \ref{p:psp}} \label{s:pfpsp}
In all cases below, the index $k\in \{1,2,3\}$ is fixed and we avoid mentioning it whenever possible. For instance, we write $q$ for $q_k$, ${\mathcal{X}}$ for ${\mathcal{X}_k}$ and $\alpha=1-\frac1q$.

\subsection{Proof of Proposition \ref{p:psp}, a.\ and b.\ parts:  lacunary case}
We recall that each $S_{P_k}f(x)= \chi_{P} T_{m_{P}} $ with $\chi_{P}\in X_{I_{P}}$ and $m_P \in M_{\omega_{P_k}}$. Preliminarily observe that due to the support conditions on $m_P$ and $\widehat{\chi_P}$,  the Fourier transforms of the functions
\begin{equation}
\label{e:ppdeflac1}
\Pi_{k,j}f=\sum_{P \in \T(j)} S_{P_k}f
\end{equation}
are supported in the disjoint intervals $\{\xi:2^{jJ-1}\leq |\xi| \leq 2^{jJ+1} \}$ where the symbol of $S_j$ is constant equal to one.
 In the lacunary case, the definition of $\Pi_k$ is then  very simple, namely referring to \eqref{e:ppdeflac1}
\begin{equation}
\label{e:ppdeflac2}
\Pi_k  \coloneqq \sum_{j\in \mathbf{J}_\T} \Pi_{k,j}
\end{equation}
and the equality in b.\ is immediate from the above considerations , while the estimate in a.\ for $p=q_k$ is immediate from the definition, as $\Pi_k$ is itself a tree operator.
 We now prove the estimate
\begin{equation}
\label{e:BMOlacbd}
\|\Pi_k f\|_{\mathrm{BMO}(\R; {\mathcal{X}_k})} \lesssim \mathsf{eng}_k(f)(\mathsf T; q_k)
\end{equation}
and a.\ for the other values of $p$ will follow by interpolation.

Fix a $J$-dyadic interval $I$. We first bound the contribution of the large scales: set $\T^+=\{P \in \T:\ell(I_P)>\ell(I)\}$. Then if $P \in \T^+$ with $\ell(I_P) =2^{v}\ell(I)$ and $\mathrm{dist}(I_P,I) \sim 2^{n} \ell (I_P)$, $v,n\in \mathbb N$, the Poincar\`e inequality yields
\[
\mathrm{osc}_I \left( S_{P_k}f  \right) \lesssim \frac{1}{|I|^{\frac1q} }\left\| \ell(I) \nabla S_{P_k}f \right\|_{L^{q}(I; {\mathcal{X}})} \leq 2^{-\alpha v}
\frac{1}{|I_P|^{\frac1q}}  \left\|\ell(I_P)  \nabla (S_{P_k}f) \right\|_{L^{q}(I; {\mathcal{X}})}.
\]
We write $\widetilde m_{P}(\xi)\coloneqq \ell(I_P) \xi m_{P}(\xi)$, so that  \[ \begin{split}
|\ell(I_P)\nabla( S_{P_k} f)|_{{\mathcal{X}}}&= |\chi_{P} T_{\widetilde m_{P}} f + (\ell(I_P) \nabla\chi_{I_P})   T_{ m_{P} } f|_{{\mathcal{X}}} \leq  \widetilde {\chi_{I_P}} \left(\left| \widetilde {\chi_{I_P}}   T_{ \widetilde m_{P}} f\right|_{{\mathcal{X}}} +  \left|\widetilde {\chi_{I_P}} T_{m_P} f\right|_{{\mathcal{X}}} \right)\\ & \coloneqq \widetilde {\chi_{I_P}}\left( |   S_{P_k,1} f|_{{\mathcal{X}}} + | S_{P_k,2} f|_{{\mathcal{X}}}\right)
\end{split}  \]  for a suitable choice of $\widetilde {\chi_{I_P}} \in {X_{I_P}}$ so that the domination of the last display holds.
Observe that with this choice $S_{P_k,u}, u=1,2$,    belong to the class $\mathbb S_{P_k}$ and are thus single scale tree operators, whence
\[
\frac{1}{|I_P|^{\frac1q}}  \left\| S_{P_k,u}\right\|_{L^{q}(\R; {\mathcal{X}})} \lesssim \mathsf{eng}_k(f)(\T; q).
\]
Using the bounds $\|\widetilde {\chi_{I_P}} \|_{L^\infty(I)} \lesssim 2^{-100n} $, we have proved that
\[
\mathrm{osc}_I \left( S_{P_k}f  \right) \lesssim 2^{-\alpha v-100n}\mathsf{eng}_k(f)(\T; q).
\]
which is summable over $P\in \T^+ $, that is over $v,n\in \mathbb N$ as claimed.

We move to handling the small scales, that is $\T^-=\{P \in \T:\ell(I_P)\leq\ell(I)\}$. We may partition $\T^-$ as the union of $\T^{-,0}=\{P \in \T^-:I_P\subset 3I\}$  and $\T^{-,n}=\{P \in \T^-:I_P\subset (2^{n+1}+1)I\setminus (2^{n}+1)I\}$ for $n\geq 1$.    We may choose $\widetilde{\chi_I} \in X_I$ so that the estimate
\begin{equation}
\label{e:BMOlacbd1}
\mathrm{osc}_I \left( \sum_{P \in \mathsf T^-}S_{P_k}f  \right) \leq \frac{1}{|I|^{\frac1q}} \sum_{n\geq 0}  \| g_n  \|_{L^q(\R; {\mathcal{X}})} \qquad g_n= \sum_{P \in \mathsf T^{-,n}}\widetilde{\chi_I} S_{P_k}f
\end{equation}
holds.  We now estimate each term appearing in the  last summation over $n$.  Fix $P\in \T^{-,n}$ for a moment and notice that  $\dist(I,I_P)\sim 2^{n} \ell(I)$ . Writing ${S_{P_k}}f = \chi_{I_P} T_{m_{P_k}} f$, define
\[
\zeta_{P}\coloneqq 2^{100n}\widetilde{\chi_I}\chi_{I_P}, \qquad  \widetilde{S_{P_k}}f = \zeta_P T_{m_{P_k}} f.
\]  We claim that the function $\zeta_{P} $ belongs to $X_{I_P}$. Indeed, the decay condition \eqref{e:frequencylocalized1} for $\zeta_P$ is easy to verify, with the additional $2^{100n}$ factor being allowed by virtue of  the previously observed separation between $I,I_P$.  The frequency support condition \eqref{e:frequencylocalized3} for $\zeta_P$ derives from the fact that the Fourier support of $ {\widetilde{\chi_I}}$ has an equal or smaller scale than the Fourier support  of $\widehat{\chi_{I_P}}$. Then we notice that $\mathsf T^{-,n}$ is a tree with top data $(I^n,0)$ and contained in $\mathsf T$, whence
\[
\| g_n  \|_{L^q(\R; {\mathcal{X}})} \leq 2^{-100n} \left\| \sum_{P \in \mathsf T^{-,n}}\widetilde{S_{P_k}}f  \right\|_{{L^q(\R; {\mathcal{X}})}} \leq
2^{-100n} |I^n|^{\frac1q} \mathsf{eng}_k(f)(\T; q) \leq 2^{-99n} |I|^{\frac1q} \mathsf{eng}_k(f)(\T; q)  \]
where the second bound holds because the operator inside the norm is a tree operator. Summation of the above bounds over $n$ yields the required control for the left hand side of \eqref{e:BMOlacbd1}. This estimate completes the proof of a.\ and b.\ parts of the Proposition.

\subsection{Proof of Proposition \ref{p:psp}, a.\ and c.\ part:  overlapping case}
We keep the convention of  writing $q$ for $q_{{\mathcal{X}_k}}$ and ${\mathcal{X}}$ for ${\mathcal{X}_k}$.

 Let $j \in \mathbf j_\T$. In this proof, we use the notation of \cite[Section 4]{MTT2} for\footnote{The collection $\mathbf{I}_\T$ is made of those $J$-dyadic intervals with the properties: 1) $3I$ does not contain any $I_P$ with $P\in \T$; 2) the $J$-dyadic parent of $I$ fails 1). Then $\tilde E_j=\cup\{I\in \mathbf{I}_\T : \ell(I) < 2^{-jJ}\}$. The set $\tilde E_j $ is a union of $J$-dyadic intervals of length $2^{-jJ}$, and obviously $E_{j+1}\subset E_j$. The collection of the connected components of $E_j$ is referred to as $\Omega_j$.} the sets $\tilde{E}_j$,  the collections $\Omega_j$, the intervals $I_j^\ell, I_j^r$, to which we send for a detailed  definition.
There is no loss in generality with assuming that $\inf \mathbf j_\T=0 $, this corresponds to the normalization $\ell(I_\T)=1$.
Define for $x\in \tilde E_0$, $j(x)=\max\{ j\in \mathbf{j}_\T: x \in \tilde E_j\}$. The scale $2^{-j(x)J}$ is the smallest spatial scale relevant for $x$. It is logical to choose $2^{j(x)J}$  as the frequency scale for the cutoff at $x$, motivating the definition of
\begin{equation}
\label{e:pitilde}
\widetilde{ \Pi_k }f\coloneqq \mathbf{1}_{\tilde E_0} T_{j(x)} f = \mathbf{1}_{\tilde E_0} T_{0} f +  \sum_{j\geq 1} \mathbf{1}_{\tilde E_j} S_j f
\end{equation}
where the nestedness of $\tilde E_j$ and telescoping have been used to get the second equality. The construction of the actual phase-space projection operator $\Pi_k$ is made by suitably modifying $\widetilde{\Pi_k} $ and begins now.

Fix a scale $j\in \mathbf{j}_\T$ and a connected component $I=[x_I^\ell,x_I^r]\in \Omega_j$. The perturbation of $g_j=\mathbf{1}_{I} S_j f $ is made by adding and subtracting    two auxiliary pieces at spatial scale $2^{-jJ}$ which kill the mean value of $g_j$: details follow.

 Recall from \cite[Lemma 4.12]{MTT2} that $I^{\ell}_j$ [resp.\ $I^r_j$] are intervals of length $2^{-2}\ell(I)$ whose right endpoint [resp.\ left endpoint] sits to the left of $x_I^\ell$ [resp.\ to the right of $x_I^r$] at a distance of $2^{-2}\ell(I)$. These intervals are   well separated, see \cite[Lemma 4.12]{MTT2} over $I \in \Omega_j, j \in \mathbf{j}_\T.$ Introduce  bump functions $\phi_{I,j}^\ell$ [resp.\ $ \phi_{I,j}^r $] adapted to $I_j^\ell$ [resp.\ adapted to $I_j^r$ ] with normalization
\[
\int \phi_{I,j}^\ell= \int \phi_{I,j}^r  = 2^{-Jj}.\]
Decomposing $\mathbf{1}_I(x) =H_I^\ell(x) + H_I^r(x)\coloneqq H(x-x_I^\ell) - H(x-x_I^r)$, where $H$ stands for Heaviside function. We introduce the $\mathcal X$-valued coefficients
\begin{equation}
\label{e:coeff}
c_{I,j}^\star\coloneqq 2^{Jj}\int H_I^\star S_j f, \qquad \star \in \{\ell, r\}.
\end{equation}
Before the next lemma,   by combining \cite[Lemmata 4.10, 4.11]{MTT2}, we realize that for $\star \in \{\ell, r\}$  there are tiles $P^\star \in \T$ with $\ell(I_{P^\star}) \sim 2^{-jJ}$ and $\dist(I_{P^\star},\{x_I^\star\})\lesssim 2^{-jJ}$.
\begin{lem}\label{l:coeff} Let  $\chi \in {X}_{I_{P^\star}}, $  $\star \in \{\ell, r\}.$ We have the estimate
\[
 |c_{I,j}^\star|_{\mathcal X} \lesssim \frac{1}{|I_{P^\star}|} \int |\chi (x)S_j f(x)|_{\mathcal X} \, \mathrm{d} x,
\]
and in particular\begin{equation}
\label{e:coeff2}
|c_{I,j}^\star|_{\mathcal X} \lesssim \mathsf{eng}_k(f)(\T; q).
\end{equation}
\begin{proof}The first inequality is proved in the same fashion as \cite[eq.\ (65)]{MTT2}. For the second, choose $\chi \in {X}_{I_{P^\star}}$ and note that $\chi^2    \in {X}_{I_{P^\star}}$ as well.  H\"older inequality yields
\[
\frac{1}{|I_{P^\star}|} |\chi^2(x)S_j f(x)|_{\mathcal X} \, \mathrm{d} x \leq \frac{1}{|I_{P^\star}|} \|\chi\|_{q'} \|\chi S_j\|_{L^q(\R; \mathcal{X})} \sim  \frac{1}{|I_{P^\star}|^{\frac1q}}   \|\chi S_j\|_{L^q(\R; \mathcal{X})}\lesssim \mathsf{eng}_k(f)(\T; q),
\]
where the last bound follows because $\chi S_j\in \mathbb S_{P^\star}$. This completes the proof.
\end{proof}
\end{lem}
With Lemma \ref{l:coeff} in hand, we are able to define the phase-space projection operator: with reference to \eqref{e:pitilde},
\begin{equation}
\label{e:pspover}
\Pi_k f \coloneqq \widetilde{ \Pi_k } f - \sum_{j \in \mathbf{j}_\T} \sum_{I \in \Omega_j} \sum_{\star \in \{\ell,r\}} c_{I,j}^\star \phi_{I,j}^\star.
\end{equation}
\subsubsection{Proof of Proposition \ref{p:psp}, part a. for $k \in B$} It suffices by interpolation to prove   estimate \eqref{e:pspa} for $p=q$ together with the endpoint
\begin{equation}
\label{e:pspaend}
\|\Pi_k f\|_{L^\infty(\R; \mathcal{X})} \lesssim \mathsf{eng}_k(f)(\T; q).
\end{equation}
\begin{proof}[Proof of \eqref{e:pspaend}] First of all, by virtue of the  separation properties of the  support of the $\phi_{I,j}^\star$ over $I \in \Omega_j, j \in \mathbf{j}_\T$ we have recalled earlier, and of the second bound in Lemma \ref{l:coeff},
\[
\left\|\sum_{j \in \mathbf{j}_\T} \sum_{I \in \Omega_j} \sum_{\star \in \{\ell,r\}} c_{I,j}^\star \phi_{I,j}^\star \right\|_{L^\infty(\R; \mathcal X)} \lesssim  \mathsf{eng}_k(f)(\T; q).
\]
Hence, it suffices to prove an $L^\infty $ bound on $\widetilde{\Pi_k}$. Fix $x\in \tilde E_0$ and set $j=j(x)$. By construction of $j(x)$ there is an interval $I'\subset \tilde E_j$ of length $2^{-jJ}$ containing $x$, and by construction of $\mathbf{I}_\T$ there is a tile $P\in \T$ with $\ell(I_P)\sim 2^{-jJ}$ and $I_P\subset 10 I'$.
Picking $\chi_{P} \in {X}_{I_P}$, writing $\alpha=\frac1q$ as before and using  Lemma \ref{l:bernstein} in the second inequality
\begin{equation}
\label{e:pspaneed}
|\widetilde{\Pi_k} f(x)|_{\mathcal X} \lesssim \chi_P |T_jf(x)|_{\mathcal X} \lesssim |I_P|^{-\alpha}\|\chi_P T_j f\|_{L^{q}(\R; \mathcal{X})} \lesssim \mathsf{eng}_k(f)(\T; q).
\end{equation}
Here the last bound comes from the fact that $g\mapsto \chi_P T_j g \in \mathbb S_P$ and $\{P\}\subset \T$ is a lacunary tree. This completes the proof of \eqref{e:pspaend}.
\end{proof}
\begin{proof}[Proof of \eqref{e:pspa} for $p=q$] First of all, using the disjointness of $I\in\Omega_j, j \in \mathbf j_\T$ we estimate the $L^q(\R; \mathcal X)$-norm of the part involving the $\phi_{I,j}^\star$ by
\[
\left(\sum_{j \in \mathbf j_\T} \#\Omega_j 2^{-jJ}  \right)^{\frac1q}
\Bigg(
\sup_{\star \in \{\ell, r \}}\sup_{\substack{j \in \mathbf j_\T\\ I \in \Omega_j \\ }} |c_{I,j}^\star|_{\mathcal X}\Bigg)\lesssim |I_\T|^{\frac1q} \mathsf{eng}_k(f)(\T; q).
\] where the first factor is bounded directly by \cite[Lemma 4.12]{MTT2} while the second is \eqref{e:coeff2} from  Lemma \ref{l:coeff}. We are then left with proving
\begin{equation}
\label{e:psparegular}
\|\widetilde{\Pi_k} f\|_{L^q(\R; \mathcal{X})}^q \lesssim |I_\T|\mathsf{eng}_k(T) (f;q)^q.
\end{equation}
To prove \eqref{e:psparegular} we recall that the sets $\tilde E_j$ are decreasing in $j$ and each is a union of disjoint  intervals $I\in \mathbf{I}_j$ with $\ell(I)=2^{-jJ}$ \cite[Lemma 4.10]{MTT2}. Thus, the sets $E_{I}=I \cap (\tilde E_j \setminus \tilde E_{j+1}), I \in \mathbf{I} $ are a disjoint cover of each $E_j\setminus E_{j+1}$, and the latter sets are also pairwise disjoint and cover the support of $\widetilde{\Pi_k} f$.
Furthermore,
\begin{equation}
\label{e:pspaboundI}
\sum_{j\in \mathbf{j}_\T} \sum_{I \in \mathbf{I}_j} |I| \lesssim |I_\T|,
\end{equation}
as $|E_I|\geq 2^{-J}|I|$. In modern terms, the  collection $\bigcup_j \mathbf{I}_j$ is $2^{-J}$ sparse.
Indeed, for each $I \in \mathbf{I}_j$ we may find $I'\subset \tilde E_j\setminus \tilde E_{j+1} $ with $I'\subset I$ and $\ell(I')=2^{-J}\ell(I)$; see \cite[Lemma 4.10]{MTT2}.
 As $\widetilde{\Pi_k} f(x) = T_j f$ for $x\in \tilde E_j \setminus \tilde E_{j+1} $,   the left hand side of \eqref{e:psparegular} is controlled by
\[
\sum_{j\in \mathbf{j}_\T} \sum_{I \in \mathbf{I}_j} \|\mathbf{1}_{E_I}T_j f\|_{L^q(\R; \mathcal{X})}^q \lesssim  \sum_{j\in \mathbf{j}_\T} \sum_{I \in \mathbf{I}_j} \|\chi_{I}T_j f\|_{L^q(\R; \mathcal{X})}^q, \qquad \chi_I \in X_I.
\]
By virtue of the last display and of \eqref{e:pspaboundI}, it suffices to show that
\[
\|\chi_{I}T_j f\|_{L^q(\R; \mathcal{X})}^q\lesssim |I| \mathsf{eng}_k(f)(\T; q)^q, \qquad \forall j\in \mathbf{j}_\T, \,I \in \mathbf{I}_j.
\]
Fix such $j, I$.
We now appeal to \cite[Lemma 4.11]{MTT2} to find $P \in \T$ with $I_P \subset 10I$ and $\ell(I_P) \lesssim 2^{-jJ}$. As $\chi_I\in \widetilde{X_{I_P}}$, $f\mapsto \chi_{I}T_j f $ belongs to $\mathbb S_P$ and thus is a tree operator for $\{P\} \subset \T$, and the last display follows, completing the proof of \eqref{e:pspa}.
\end{proof}
\subsubsection{Proof of Proposition \ref{p:psp}, part c.} We begin  the proof by using the   single scale estimate of Lemma \ref{l:sscale}.
In fact \eqref{e:pspc2} follows immediately from \eqref{e:pspc}, \eqref{e:pspc2a} and the fact that $\mu_j$ is uniformly bounded. So it remains to prove \eqref{e:pspc}. As usual, we prove the extremal cases. In fact, it suffices to prove the case  $p=q$, as the case $p=\infty$ may then be recovered from Lemma \ref{l:bernstein}.
\begin{proof}[Proof of \eqref{e:pspc} for $p=q$ ]
In the proofs that follow, we use the local notation
\begin{equation}
\label{e:localnot}
S_{P_k} g = \chi_P T_{m_{P_k}} g, \qquad
O g= \sum_{  P \in \T(j_0)  } S_{P_k} g,
\end{equation}
where $\chi_P \in X_{I_P}= X_{I_P}(2N,\delta,C,c)$ and $m_{P_k}\in M_{\omega_{P_k}}$.
Notice that $O$ is a  tree operator and as thus is bounded on  $L^{q}(\R; \mathcal{X}_k) $, but it is also pointwise bounded by maximal averages and thus  bounded on $L^{\infty}(\R; \mathcal{X}_k) $.


The first step in the proof proper is to notice that \[
O ( f -\Pi_k f)= O( T_{j_0} f -\Pi_k f),\] leading to the  key is the decomposition
\begin{align}
\nonumber  T_{j_0} f -\Pi_k f&=
\\  &\quad  \mathbf{1}_{\R  \setminus \tilde E_{j_0}}  T_{j_0} f  \label{e:pspdec1}
\\ & -\mathbf{1}_{\R  \setminus \tilde E_{j_0}} \widetilde{\Pi_k } f  \label{e:pspdec2}
\\ & + \sum_{\star \in \{\ell,r\}} \mathbf{1}_{\R  \setminus \tilde E_{j_0}} \sum_{j \leq j_0} \sum_{I\in \Omega_j} c_{I,j}^\star \phi_{I,j}^\star \label{e:pspdec3} \\ &
-\sum_{\star \in \{\ell,r\}}  \sum_{j>j_0} \sum_{I\in \Omega_j} (H_I^\star S_j f- c_{I,j}^\star \phi_{I,j}^\star). \label{e:pspdec4}
\end{align}
cf.\ \cite[eqs. (77)-(82)]{MTT2}. We now have to estimate the four contributions separately, and, as in \cite{MTT2}, distinguish the local case $5I_0 \cap  \tilde E_{j_0} \neq \emptyset $ from the complementary nonlocal case: for clarity, we first present the local case, and at the end of the proof we elaborate on the sketch provided in \cite[p.\ 295]{MTT2} and unify the two cases: see Remark \ref{rem:nonl} below.

We first estimate the contribution of
$g=\eqref{e:pspdec1} - \eqref{e:pspdec2} + \eqref{e:pspdec3}$.
Using the decay at scale $\ell(I_0)$ of the kernel of $O$ together with  the $L^\infty$ bounds \eqref{e:pspaneed}, \eqref{e:pspaend}, \eqref{e:coeff2}
\begin{equation}
\label{e:pspdec5}
\|\mathbf{1}_{I_0}O (\mathbf{1}_{\R\setminus 3I_0}g)\|_{L^{q}(\R; {\mathcal{X}_k})} \lesssim |I_0|^{\frac1q} \sum_{ j \leq j_0} 2^{-J|j-j_0|} \left\langle \frac{\dist(I_0, \partial \tilde E_{j})}{ \ell(I_0)}\right\rangle^{-100}  \mathsf{eng}_k(f)(\mathsf T;q)
\end{equation}
which is acceptable for \eqref{e:pspc}. Further, if $\mathbf{1}_{ 3I_0}g$ is nonzero, then $I_0$ is close to the boundary of $\tilde E_{j_0}$. In this case the right hand side of \eqref{e:pspc} is $O(1)$ and we may just aim for the estimate
\begin{equation}
\label{e:forg1}
  \|O (\mathbf{1}_{ 3I_0}g)\|_{L^{q}(\R; {\mathcal{X}_k})}    \lesssim |I_0|^{\frac1q} \|O\|_{L^{q}(\R; {\mathcal{X}_k})} \mathsf{eng}_k(f)(\mathsf T;q).
\end{equation}
Although the $O$-norm appearing here is $O(1)$, we choose to keep this constant in evidence for later use.

 We begin the proof of \eqref{e:forg1}. We argue   separately for each summand of $g$.
First of all, we bound the contribution of \eqref{e:pspdec1}. Appealing to \cite[Lemma 4.11]{MTT2}, we learn that there exists $P \in \mathsf{T}(j_0)$ such that $\dist(I_0,I_P) \sim 1$, so that for suitable choice of $\chi_{I_0}\in X_{I_0}$,
\begin{equation}
\label{e:Zproof1}
\|\mathbf{1}_{3I_0}  T_{j_0} f\|_{L^q(\R; \mathcal X_k)}   \leq
\|\chi_{I_0}  T_{j_0} f\|_{L^q(\R; \mathcal X_k)}  \lesssim  |I_0|^{\frac1q} \mathsf{eng}_k(f)(\mathsf T;q).
\end{equation}
This makes the contribution of \eqref{e:pspdec1} acceptable for \eqref{e:forg1}. To control the contribution of \eqref{e:pspdec2} we note that  $(\mathbb R\setminus\tilde E_{j_0}) \cap 3I_0$ is the union of at most three intervals $I_1 $ of length $\ell(I_0)$, on which $\widetilde {\Pi_k} f$ coincides with $T_{j_0-1} f$. On each of these intervals, by the same argument used for \eqref{e:Zproof1},
\begin{equation}
\label{e:Zproof2}
\|\mathbf{1}_{I_1}  T_{j_0-1} f \|_{L^q(\R; \mathcal X_k)}   \lesssim
 |I_0|^{\frac1q} \mathsf{eng}_k(f)(\mathsf T;q)
 \end{equation} which is acceptable. Finally,
 from  the last claim of \cite[Lemma 4.12]{MTT2} we gather that  $I_{j}^\star \cap 3_{I_0}\neq \varnothing $ for at most $O(1)$ intervals $I\in \Omega_j$ with $j\leq j_0$. Therefore
\[
\|\mathbf{1}_{ 3I_0}\eqref{e:pspdec3}\|_{L^{q}(\R; {\mathcal{X}_k})} \lesssim |I_0|^{\frac1q} \sup_{I,j,\star} |c_{I,j}| \lesssim |I_0|^{\frac1q} \mathsf{eng}_k(f)(\mathsf T;q)
\]
by \eqref{e:coeff2}, and we have proved \eqref{e:forg1}. This finishes the control of terms \eqref{e:pspdec1} to \eqref{e:pspdec3}.

To complete the proof of \eqref{e:pspc}, we are left with estimating the small spatial scales term \eqref{e:pspdec4}. Using the triangle inequality and the definition of $\mu_{j_0}$, it will suffice to prove that for each fixed $\star\in \{\ell, r\}, j>j_0, I\in \Omega_j$ there holds
\begin{equation}
\label{e:pspc40}
\begin{split} &
\|\mathbf{1}_{ I_0}O(G_{I}) \|_{L^{q}(\R; {\mathcal{X}_k})} \lesssim |I_0|^{\frac1q} \mathsf{eng}_k(f)(\T; q) \int \frac{\chi_{I_0}(x)}{|I_0|} 2^{-\frac{(j-j_0)}{100} }
\langle 2^{jJ}|x-x^\star_I|\rangle^{-100} \, \mathrm{d} x,
\\
& G_{I}\coloneqq H_I^\star S_j f- c_{I,j}^\star \phi_{I,j}^\star.
\end{split}
\end{equation}
As they will be kept fixed below, we have omitted $\star $ and $j$ from the $G_I$ notation for simplicity.
Let $n\in \mathbb N$ be the least integer such that $2^{n}I_0\cap I_{j}^\star \neq \varnothing $.
A direct  computation of the right hand side and the fact that $\chi_{I_0}\in X_{I_0} $ tells us that the above bound is equivalent to the estimate
\begin{equation}
\label{e:pspc4}
\|\mathbf{1}_{ I_0}O(G_{I}) \|_{L^{q}(\R; {\mathcal{X}_k})} \lesssim |I_0|^{\frac1q} 2^{-\frac{(j-j_0)}{100} } 2^{-jJ} 2^{-100n}.
 \end{equation}

 The final stretch of the proof will be to establish \eqref{e:pspc4}.
As the frequency support of $O$ is localized near $2^{j_0}$, we gather that $O[(T_{j-1} H_I^{\star})(S_jf)]=0$. This means we may replace $G_I $ by
\[
F_I=G_I-(T_{j-1} H_I^{\star})(S_jf) = [(1-T_{j-1})H_I^{\star}]S_j f - c_{I,j}^\star \phi_{I,j}^\star.
\]As both $G_I$ and $F_I-G_I$ have mean zero, $F_I $ also does. Letting $\Phi_I$ be the antiderivative of $F_I$ which vanishes at $\pm\infty$, we have
\[
OF_I (x) = 2^{Jj_0} \widetilde O \Phi(x),
\]
where $\widetilde O$ is the pseudodifferential operator
\[
\widetilde O g(x) = \sum_{  P \in \T(j_0) }  \int \chi_{I_P}(x) 2^{-Jj_0}\xi m_{P_k}(\xi) \widehat g(\xi) \e^{ix\xi }\, \mathrm{d} \xi.
\]
In fact $\widetilde O$ is a tree operator as $\xi \mapsto 2^{-Jj_0}\xi m_{P_k}(\xi) $ belongs to $M_{\omega_{P_k}}$. Therefore, we begin to bound \eqref{e:pspc2} by
\begin{equation}
\label{e:pspc5}
\|\mathbf{1}_{ I_0}OG_{I}) \|_{L^{q}(\R; {\mathcal{X}_k})} \lesssim \|\chi_{I_0}\Phi_I  \|_{{L^{q}(\R; {\mathcal{X}_k})}}.
\end{equation}
An estimate on  $|\Phi_I(x) |_{\mathcal X} $ compatible with the right hand side of \eqref{e:pspc4} may be produced, cf. \cite[p. 298]{MTT2}, once  we establish the pointwise bound
\begin{equation}
\label{e:pspc6}
|F_I(x)|_{\mathcal X} \lesssim  \mathsf{eng}_k(f)(\T; q)  \langle 2^{jJ}|x-x^\star_I|\rangle^{-100}.
\end{equation}
The last step towards  \eqref{e:pspc4}, and therefore \eqref{e:pspc}, is to prove \eqref{e:pspc6}. The contribution of $ c_{I,j}^\star \phi_{I,j}^\star$ is controlled by virtue of the decay of $\phi_{I,j}^\star$ and \eqref{e:coeff2}. We turn to controlling the summand $[(1-T_{j-1})H_I^{\star}]S_j f $. First we recall that by construction of $I\in \Omega_j$ and $I_{j}^\star,$ we may find a dyadic interval $I'$ with $\ell(I')=2^{-jJ} $ adjacent to the left endpoint of $I$, in particular $\dist(x_I^\ell, I')\sim 2^{-j} $, and    $P\in \T$ such that $I_P \subset I'$. Pick $\chi' \in {X}_I$ with decay parameter $N$ (for instance). As $g\mapsto \chi_{I'}S_j g\in \mathbb S_P$ and $\{P\}\subset \T$, we have the estimate
\[
\|\chi_{I'} S_j f\|_{L^\infty(\R;\mathcal X)} \leq \frac{1}{|I'|^{\frac1q}} \|\chi_{I'} S_j f\|_{L^q(\R;\mathcal X)} \lesssim \mathsf{eng}_k(f)(\T; q)
\]
where the first inequality is Lemma \ref{l:bernstein}. As $\dist(x_I^\ell, I')\sim 2^{-j} $, we have that $\chi_{I'}(x) \gtrsim \langle 2^{jJ}|x-x_{I_j^\star}|\rangle^{-N}$. Putting these estimates together,
\[
|S_j f(x) | \lesssim \langle 2^{jJ}|x-x_{I_j^\star}|\rangle^N\mathsf{eng}_k(f)(\T; q).
\]
Integrating repeatedly by parts the high frequency  function $[(1-T_{j-1})H_I^{\star}]$, we may bound it pointwise by factors of $\lesssim_N \langle2^{jJ}|x-x^\star_I|\rangle^{-N-100}$, compensating the polynomial growth of the last display and yielding an acceptable right hand side for \eqref{e:pspc6} which is finally proved. The proof of   \eqref{e:pspc} is finally complete.
\end{proof}
\begin{rem}[The nonlocal case of \eqref{e:pspc}]  \label{rem:nonl}The local/nonlocal cases can be unified by introduction of the parameter
\[
Z= \textrm{least nonnegative integer such that } I_0 \pm Z \ell(I_0) \cap  \tilde E_{j_0}\neq\emptyset.
\]
Comparing with what we did to obtain \eqref{e:Zproof1}, and to \cite[Lemma 4.11]{MTT2}, we learn that there exists $P \in \mathsf{T}(j_0)$ such that $\dist(I_0,I_P) \sim Z$, so that
\begin{equation}
\label{e:Zproof3}
\mathsf{eng}_k(f)(\mathbb P(I_0);q) \lesssim Z^{N}, \qquad
\mathbb P(I_0)=\{P \textrm{ is any tri-tile with } I_P=I_0\}.
\end{equation}This introduces a $Z^N$ loss in e.\ g.\ estimates \eqref{e:Zproof1}, \eqref{e:Zproof2}.
However, as we are concerned with estimates for $1_{I_0}O(T_{j_0} f -\Pi_k f)$, we may replace $O$ by the operator $g\mapsto \tilde O g= \chi_{I_0} Og $, where $\chi_{I_0}\in X_{I_0}(2N,\delta, C,c)$ and  $\chi_{I_0}\geq \mathbf{1}_{I_0}$. The separation between $\tilde E_{j_0}$ and $I_0$ yields that
\begin{equation}
\label{e:Zproofk}
\|\tilde O \|_{L^q(\R; \mathcal X_k)}   \lesssim Z^{-2N},
\end{equation}
and the same additional decay factor is gained in the kernel estimates for $\tilde O$.
Replacing  $O$ by $\tilde O$ in \eqref{e:pspdec5}, \eqref{e:forg1} and taking \eqref{e:Zproofk} into account offsets the loss introduced by \eqref{e:Zproof3}.
\end{rem}

\section{Proof of Lemma \ref{l:energy}} \label{s:energyproof}
We begin with a definition. We say that a family of trees $\T\in \mathbf{T}$ is $k$-strongly disjoint with parameter $\beta\geq 1$ if
\begin{itemize}
\item[i.] each $\T$ is a $k$-lacunary tree;
\item[ii.] if  $ \T,\T'\in\mathbf{T},$ $\T\neq\T'$,  then \[
P\in \T,\,P'\in \T', \, \ell(\omega_{P}) \leq \ell(\omega_{P'}), \, 10\beta\omega_{P_k}\cap10\beta \omega_{P'_k} \neq \varnothing \implies I_{P'} \cap I_{\T} =\varnothing.  \]
\end{itemize} The rationale behind this definition is that, if the consequence of the above implication failed, the tri-tile $P'$ would qualify to be in a suitable completion of the tree $\T$.
In what follows, we work with the parameter $\beta=1$, as the general case   $1\leq \beta \ll 2^{J}$  may be handled by finite splitting.
\subsection{The $L^2$-orthogonality estimates}
Tree operators associated to families of $k$-strongly disjoint trees give rise to an $L^2$ almost orthogonality estimate: this is well known, and extends to the case of Hilbert space valued functions, as detailed in the next lemma. This lemma is a transposition of \cite[Proposition 6.1]{HLC} to our context. It is convenient in what follows to introduce the single tile  version of the energy parameters. To do so, for each tri-tile $P$ we introduce the function
\begin{equation} \label{e:up2}
\begin{split}
u_{P}\coloneqq \left\langle \frac{|x-c(I_P)|}{\ell(I_P)}\right\rangle^{-10}\end{split}
\end{equation}
and define
\begin{equation}
\label{e:sscaleen}
\begin{split}& \|f\|_{P,k,q}= \sup_{m_{P_k} \in M_{P_k}}  {|I_{\mathsf{P}}|^{-\frac1q}} \left\| u_P T_{m_{P_k}} f \right\|_{L^q(\R; \mathcal{X})},
\\
&
\mathsf{eng}_{k,ss}(f) (\mathbb P;q) \coloneqq \sup_{P \in \mathbb P} \|f\|_{P,k,q}.\end{split}
\end{equation}
For uniformity, we gave the definitions above for a generic $1\leq q \leq \infty$. However, Lemma \ref{l:bernstein} shows the upper bound  $ \|f\|_{P,k,\infty}\lesssim \|f\|_{P,k,1}$, and it follows that $\|f\|_{P,k,p} \sim_{p,q} \|f\|_{P,k,q} $ for all $1\leq p ,q\leq \infty$. Below, we will only use the value $q=2$ in \eqref{e:sscaleen}.
Note the trivial bounds
\begin{align}
\label{e:sscaleentv}
 & \|f\|_{P,k,q}\lesssim |I_{P}|^{-\frac1q}  \|u_{P} \|_q  \sup_{m_{P_k} \in M_{P_k}} \| T_{m_{P_k}} f \|_{L^\infty(\R; \mathcal X)} \lesssim \|f\|_{L^\infty(\R; \mathcal X)}, \\ &
\label{e:sscaleentv2}
\sup_{S_{P_k} \in \mathbb S_{P_k}}  \|S_{P_k} f\|_{L^q(\R; \mathcal{X})}\lesssim |I_P|^{\frac1q}\|f\|_{P,k,q}. \end{align}
\begin{lem} \label{l:ortho2} Let $\mathcal X$ be a Hilbert space and $\mathbf T$ be a collection of $k$-strongly disjoint trees, and define $\mathbb{T}= \bigcup\{\T: \T \in \mathbf{T}\}$.
 There holds
\begin{equation}
\label{e:orthohilb}
\left\|\sqrt{|I_\T|}\mathsf{eng}_{k}(f)(\mathsf T; 2) \right\|_{\ell^2(\T\in \mathbf T)} \lesssim \|f\|_{L^2(\R; \mathcal X)} +
\left( \mathsf{eng}_{k,ss}(f) (\mathbb{T} ;2) \left[\sum_{\T \in \mathbf T}|I_\T|\right]^{\frac12} \right)^{\frac13} \|f\|_{L^2(\R; \mathcal X)}^{\frac23}.
\end{equation}
\end{lem}
Before the proof proper, we enucleate the almost orthogonality of the single tile operators within a $k$-lacunary tree.
\begin{lem} \label{l:ortho0} Let $\mathcal X$ be a Hilbert space and $\T$  be a $k$-lacunary tree. Then
\[
\sqrt{|I_\T|}\mathsf{eng}_{k}(f)(\mathsf T; 2)\lesssim\sqrt{\sum_{P\in \T}  |I_P| \|f\|_{P,k,2}^2
}.
\]
\end{lem}

\begin{proof}By modulation invariance, it suffices to take care of the case $\xi_\T=0$.
  Choose a tree operator $T_\T=\sum_{P\in \T} S_{P_k}$ that nearly achieves the supremum in $\mathsf{eng}_{k}(f)(\mathsf T; 2)$   and write $  S_{P_k}g= \chi_{P} T_{m_{P_k}}g$.
From the disjointness of the frequency supports, we have that, referring to \eqref{e:Tj}
\[
\langle {S}_{P_k} f, {S}_{P'_k} f \rangle \neq 0 \implies P, P'\in \T(j)
\]
For $n\in \mathbb Z$ denote by $P^{+n}$ the (at most) unique tri-tile $P'\in \T(j)$ with $I_{P'} = I_P + n \ell(I_P) $. Then define
\begin{equation} \label{e:up2}
\begin{split}
\tilde \chi_{P}\coloneqq \frac{\chi_P}{u_P},
\qquad \tilde S_{P_k} g \coloneqq \tilde \chi_{P} T_{m_{P_k}}g.
\end{split}
\end{equation}
It is immediate to see that $\tilde \chi_P \in \Psi_{I_P}$ as multiplying by a polynomial does not change the frequency support neither significantly alters the rapid decay of $  \chi_{P}$, hence $\tilde S_{P_k} $ belongs to $\mathbb S_{P_k}$.
Therefore
\[
\begin{split}
|I_\T|\mathsf{eng}_{k}(f)(\mathsf\T; 2)^2 &\lesssim
\|T_\T f\|_{L^2(\R; \mathcal X)}^2   \lesssim \sum_{j\in \mathbf{j}_\T} \sum_{P\in \T(j)} \sum_{n \in \mathbb Z} \int  {S}_{P_k} f
\overline{{S}_{P^{+n}_k} f} \\ &\leq  \sum_{j\in \mathbf{j}_\T} \sum_{P\in \T(j)} \sum_{n \in \mathbb Z} \int
|{\tilde S}_{P_k} f| |{{\tilde S}_{P^{+n}_k} f} | u_P u_{P^{+n} } \\ & \lesssim   \sum_{j\in \mathbf{j}_\T} \sum_{P\in \T(j)} \sum_{n \in \mathbb Z} \langle n \rangle^{-10}  \left(
\|{\tilde S}_{P_k} f\|_{L^2(\R; \mathcal X)}^2 + \| {{\tilde S}_{P^{+n}_k} f}\|_{L^2(\R; \mathcal X)}^2 \right) \\ &\lesssim \sum_{P\in \T}
\|{\tilde S}_{P_k} f\|_{L^2(\R; \mathcal X)}^2   \lesssim \sum_{P\in \T}  |I_P| \|f\|_{P,k,2}^2
\end{split}
\]
and this proves the claimed inequality. We have used $ \|u_P u_{P^{+n} }\|_{\infty}\lesssim \langle n \rangle^{-10}$ to pass to the second line and \eqref{e:sscaleentv2} in the last estimate.
\end{proof}
\begin{proof}[Proof of Lemma \ref{l:ortho2}] Let us choose the scaling $\|f\|_{L^2(\R; \mathcal X)}=1$.
From Lemma \ref{l:ortho0}, we may bound the quantity
\[
S\coloneqq  \sqrt{\sum_{P\in \mathbb{T}}  |I_P| \|f\|_{P,k,2}^2}
\]
in place of the  left hand side of \eqref{e:orthohilb}. Then
\begin{equation}
\label{e:ortho20}
S^2\sim \sum_{P \in \mathbb T} |I_P| \|f\|_{P,k,2}   \frac{\| u_P{  T}_{m_{P_k}} f\|_{L^2(\R; \mathcal X)}}{|I_P|^{\frac12}}
\end{equation}
having linearized the suprema in $\|f\|_{P,k,2}$ with a suitable choice ${m}_{P_k} \in   M_{P_k}$, \ $P \in  \mathbb T$. From now on, as $m_{P_k}$ and $k$ are fixed, we simply write $T_P$ in place of ${  T}_{m_{P_k}}$, and by using Lemma \ref{l:bernstein}, identifying $\mathcal X'$ with $\mathcal X$ via Riesz representation, we have
\[
\frac{\| u_P{  T}_{P} f\|_{L^2(\R; \mathcal X)} }{|I_P|^{\frac12}}  = \frac{1}{|I_P|} \langle {  T}_{P}^* v_P, f  \rangle, \qquad P \in \mathbb T
\]
for some function $v_P$ with
\begin{equation}
\label{e:vp}
|v_P(x)|_{\mathcal X} \lesssim u_P(x), \qquad x\in \R.
\end{equation}
These considerations lead to the estimate
\begin{equation}
\label{e:ortho21}
S^2\sim \left\langle \sum_{P \in \mathbb T}  \|f\|_{P,k,2}   T_{P}^* v_P, f \right\rangle     \leq \left\|   \sum_{P \in \mathbb T}  \|f\|_{P,k,2}   T_{P}^* v_P \right\|_{L^2(\R; \mathcal X)}.
\end{equation}
  Define now
\[\mathbb T_<(P)\coloneqq\{P'\in \mathbb T : \omega_{P_k} \subsetneq \omega_{P'_k} \}, \qquad \mathbb T_=(P)\coloneqq\{P'\in \mathbb T : \omega_{P'_k} =\omega_{P_k} \}. \]  Frequency support considerations applied to the inner products $\langle T_{P}^* v_P,   T_{P'}^* v_{P'} \rangle$  then lead to the chain of inequalities
\begin{equation}
\label{e:ortho22}
\begin{split} &\quad
\left\|     \sum_{P \in \mathbb T}   \|f\|_{P,k,2}   T_{P}^* v_P\right\|_{L^2(\R; \mathcal X)}^2\\
& = \sum_{P \in \mathbb T } \sum_{P' \in \mathbb T_=(P)}  \|f\|_{P,k,2}  \|f\|_{P',k,2}  \langle T_{P}^* v_P,   T_{P'}^* v_{P'} \rangle \\ &+ 2 \sum_{P \in \mathbb T } \sum_{P' \in \mathbb T_<(P)}   \|f\|_{P,k,2}   \|f\|_{P',k,2}  \langle T_{P}^* v_P,   T_{P'}^* v_{P'} \rangle  \coloneqq S_1 + 2S_2.
\end{split}
\end{equation}
We first treat $S_1$.
Note that if  $P'\in \mathbb P_=(P)$ then $P'=P^{+n}$ for some $n\in \mathbb Z$, see the line before \eqref{e:up2} for a definition. The decay of $v_P$ \eqref{e:vp} and  the kernel estimate for $T_P^*$ guarantee the pointwise bound
\begin{equation}
\label{e:ortho231}
|T^*_Pv_P|_\mathcal X  \lesssim u_P
\end{equation}
whence
\begin{equation}
\label{e:ortho232}
 | \langle T_{P}^* v_P,   T_{P^{+n}}^*v_{P^{+n}} \rangle| \lesssim |I_P|\langle n \rangle^{-10}, \qquad n \in \mathbb Z.
\end{equation}
Therefore, we control
\begin{equation}
\label{e:ortho23}
S_1 \lesssim
\sum_{P\in \mathbb P} \sum_{n \in \mathbb Z} \langle n \rangle^{-10} \left(
|I_P|\|f\|_{P,k,2}^2 +  |I_{P^{+n}}|\|f\|_{P^{+n},k,2}^2\right) \lesssim S^2
\end{equation}
using the definition of $S$. We turn to  $S_2$. Notice that if $P' \in \mathbb T_<(P), $ then $\ell(I_{P'})< \ell(I_{P}) $.  Relying on \eqref{e:ortho231} again
\begin{equation}
\label{e:ortho24}
 | \langle T_{P}^* v_P,    T_{P'}^*v_{P'} \rangle| \lesssim |I_{P'}| \left\langle \frac{\dist(I_P,I_{P'})}{\ell(I_P)} \right\rangle^{-10} \lesssim \|\mathbf{1}_{I_{P'}} u_P\|_1, \qquad P'\in \mathbb T_<(P).
\end{equation}
Also note  that, due to the  condition ii.\ in the definition of the strongly disjoint trees, if $P'\in \mathbb T_<(P)$ then $I_{P'}\cap I_{\T(P)}=\varnothing$ where ${\T(P)}$ is the unique tree in $\mathbb T$ where $P$ belongs, and furthermore if $P',P''\in \mathbb T_<(P) $ then $I_{P'}\cap I_{P''}=\emptyset$: this is a combination of conditions i.\ and ii. of the definition, see \cite{HLC,MTT2} for details.
 Using definition \eqref{e:sscaleen}, estimate \eqref{e:ortho24}, the trivial estimate $\|\mathbf{1}_{I_{P'}} u_P\|_1\leq \|u_P\|_1\lesssim |I_P|$, Cauchy-Schwarz, disjointness and separation from $I_{\T(P)}$ of $\{I_{P'}:  P'\in \mathbb T_<(P)\}$, and    we obtain
\begin{equation}
\label{e:ortho26}
\begin{split}
S_2 & \lesssim  \mathsf{eng}_{k,ss}(f) (\mathbb{T} ;2)    \sum_{P\in \mathbb T}    \|f\|_{P,k,2}  \sum_{P'\in \mathbb T_<(P)}\|\mathbf{1}_{I_{P'}} u_P\|_1 \\
&\le
\mathsf{eng}_{k,ss}(f) (\mathbb{T} ;2)    \sum_{P\in \mathbb T}   \sqrt{|I_P|} \|f\|_{P,k,2}   \|\mathbf{1}_{\R\setminus I_{\T}}  u_P\|_1^{\frac12}
\\ & \lesssim
\mathsf{eng}_{k,ss}(f) (\mathbb{T} ;2) \left( \sum_{P\in \mathbb T } |I_P|  \|f\|_{P,k,2}^2 \right)^{\frac12} \left(  \sum_{\T \in \mathbf T} \sum_{P \in \T}     \|\mathbf{1}_{\R\setminus I_{\T}} u_P\|_1  \right)^{\frac12}
\\ & \lesssim \mathsf{eng}_{k,ss}(f) (\mathbb{T} ;2) \left( \sum_{\T \in \mathbf T} |I_\T|  \right)^{\frac12} S;     \end{split}
\end{equation}
we omitted some of the details, see e.g.\ \cite[Proposition 6.1]{HLC}. Summarizing \eqref{e:ortho21}, \eqref{e:ortho22}, \eqref{e:ortho23}, \eqref{e:ortho26}
\[
S^2\lesssim \sqrt{S_1 + 2S_2} \lesssim \left( S^2+ \mathsf{eng}_{k,ss}(f) (\mathbb{T} ;2) \left( \sum_{\T \in \mathbf T} |I_\T|  \right)^{\frac12} S\right)^{\frac12}
\]
which yields the claimed bound. The details can be read from \cite[Proposition 6.1]{HLC}, hence we omit them.
\end{proof}
\subsection{Transporting almost orthogonality to interpolation spaces}In the previous subsection, we have shown that the definitions \eqref{e:engdef}, \eqref{e:sscaleen} lead to Hilbert space valued orthogonality estimates for families of strongly disjoint trees. The point is that the definitions \eqref{e:engdef}, \eqref{e:sscaleen} are of maximal nature and involve more general operators  than the rank 1 projections $f\mapsto \langle f,\varphi_{P_k}  \rangle \varphi_{P_k} $ of \cite{HLC}, namely operators of the class $\mathbb S_{P_k}$. It is because of this additional generality that we had to reproduce, with small changes, the classical $TT^*$ arguments of \cite{HLC}.

Now that our version of \cite[Prop.\ 6.1]{HLC}, namely Lemma \ref{l:ortho2} is in place, the   interpolation arguments of \cite[Section 7]{HLC} may be perused \emph{mutatis mutandis}, leading to the following almost orthogonality estimate for interpolation spaces.
\begin{prop} \label{p:int}
Let $2\leq p<\infty$ and $\mathcal X=[\mathcal Y_0, \mathcal Y_1]_{\frac2p}$ be the complex  interpolation space of a $\UMD$ space $\mathcal Y_0$ and a Hilbert space  $\mathcal Y_1$. Then for all $0<\alpha\leq 1$ the inequality
\begin{equation}
\label{e:interp1}
\left\| |I_\T|^{\frac1p}\mathsf{eng}_{k}(f)(\mathsf T; p) \right\|_{\ell^p(\T\in \mathbf{ T})} \lesssim_\alpha \|f\|_{L^p(\R; \mathcal X)} + \left( \|f\|_{L^\infty(\R; \mathcal X)} \left[\sum_{\T \in \mathbf T}|I_\T|\right]^{\frac1p} \right)^{1-\alpha}\|f\|_{L^p(\R; \mathcal X)}^\alpha
\end{equation}
holds uniformly over all collections $\mathbf T$ of $k$-strongly disjoint trees.
\end{prop}
\begin{proof} The first step of the proof consists of deducing the case $p=2$ of \eqref{e:interp1} from Lemma \ref{l:ortho2}. With
\[
\mathsf{eng}_{k,ss}(f) (\mathbb{T} ;2) \lesssim \|f\|_{L^\infty(\R; \mathcal X)}
\]
in hand, a consequence of \eqref{e:sscaleentv}, this is accomplished  following step by step the proof of \cite[Proposition 6.6]{HLC}.

The second step consists in the deduction of an endpoint at $p=\infty$, which is
\begin{equation}
\label{e:interp2}
\left\|  \mathsf{eng}_{k}(f)(\mathsf T; \mathrm{BMO}) \right\|_{\ell^\infty(\T\in \mathbf{ T})} \lesssim \|f\|_{L^\infty(\R; \mathcal X)}
\end{equation}
having denoted
\[
\|f\|_{k,\T, \mathrm{BMO}}\coloneqq
 \sup
 \left\| \mathrm{Mod}_{-\xi_\T}  T_{\mathsf T} f \right\|_{\mathrm{BMO}(\R; \mathcal{X})}, \qquad  \mathsf{eng}_{k}(f)(\mathsf T; \mathrm{BMO})\coloneqq
 \sup_{\substack{\mathsf T \subset \mathbb P \\ \T \, k-\textrm{lacunary }}}  \|f\|_{k,\T, \mathrm{BMO}}
\]
where $\mathrm{Mod}_{-\xi} $ stands for modulation by $-\xi$, and usual the  first supremum is taken over all possible choices of type $k$  tree operators  $T_{\mathsf T}$.  \
The estimate \eqref{e:interp2} is an immediate consequence of the uniform estimate for demodulated tree operators \[\mathrm{Mod}_{-\xi_\T}T_\T \mathrm{Mod}_{\xi_\T}: L^\infty(\R; \mathcal X)\to \mathrm{BMO}(\R; \mathcal{X}),\] which holds by virtue of the fact that each operator $\mathrm{Mod}_{-\xi_\T}T_\T \mathrm{Mod}_{\xi_\T}$ is a Calder\'on-Zygmund operator.

Finally, the proof of the proposition is obtained by complex interpolation of the case $q=2$ of \eqref{e:interp1}  with \eqref{e:interp2}. Details are given in \cite[Proposition 7.3]{HLC}.
\end{proof}
\subsection{The proof proper of Lemma \ref{l:energy}} The proof proceeds via an iterative algorithm similar to \cite[Proposition 8.4]{HLC}. One additional remark necessary here is that the selected trees come from a greedy selection process, and therefore satisfy properties a. b.\ and c.\ of Subsection \ref{ss:trees}, cf.\ \cite[Lemmata 4.4 and 4.7]{MTT2}.

For the proof, write $p= q_{\mathcal X}$, $\lambda \coloneqq \mathsf{eng}_{k}(f)(\mathbb P; p) $ and let $\alpha\in (0,1)$ be chosen so that $q=p/\alpha$.  Performing an iterative algorithm analogous to  \cite[Lemma 7.7]{MTT}, we  decompose
\[
\mathbb P\coloneqq \mathbb P^- \cup \mathbb P^+,
\]
where \[ \mathsf{eng}_{k}(f)(\mathbb P^-; p)\leq\frac\lambda 2 \] and  $ \mathbb P^+ = \bigcup\{\T: \T\in \mathbf T\}$ is a family of greedily selected trees with the following property: for each $\T$ there exists a $k$-lacunary tree $\T'\subset \T$ with $I_{\T}= I_{\T'}$ and the family $ \mathbf{T}'= \{\T': \mathsf T \in \mathbf T\}$  consists of $k$-strongly disjoint trees with
\[
\mathsf{eng}_{k}(f)(\mathsf T; p)   \gtrsim \lambda.
\]
Using Proposition \ref{p:int} in the second inequality,
\[
\begin{split}
\lambda^p\sum_{\T \in \mathbf T} |I_\T|&\lesssim  \left\| |I_\T|^{\frac1p}\mathsf{eng}_{k}(f)(\mathsf T; p) \right\|_{\ell^p(\T\in \mathbf{ T})}^p \lesssim
\|f\|_{L^p(\R; \mathcal X)}^p +   \|f\|_{L^\infty(\R; \mathcal X)}^{p(1-\alpha)} \left(\sum_{\T \in \mathbf T}|I_\T|\right)^{1-\alpha}  \|f\|_{L^p(\R; \mathcal X)}^{\alpha p}
\\ & \lesssim |F| + \left(\sum_{\T \in \mathbf T}|I_\T|\right)^{1-\alpha} |F|^\alpha
\end{split}
\]
Dividing into cases depending on whether $|F|$ the summand in the last line is larger or not than the  $|F|^\alpha$ one,
\[
\sum_{\T \in \mathbf T} |I_\T| \lesssim \max\{ \lambda^{-p},\lambda^{-q}\}|F|
\lesssim \lambda^{-q} |F|
\]
which is what we had to prove to conclude Lemma \ref{l:energy}. In the last comparison we have used that $q>p$ and
\[
\lambda \lesssim \sup_{P \in \mathbb P} \inf_{I_P} \mathrm{M}(|f|_\mathcal{X}) \lesssim 1,  \]
a consequence of Lemma \ref{l:engbd}. The proof of Lemma \ref{l:energy} is complete.

\section{Proof of Lemma \ref{l:engbd}} \label{s:engbdpf}
Throughout this proof,  if $I$ is a $J$-dyadic interval, we write $I^{+v}=I+ v\ell(I)$  for $v\in \mathbb Z$ to denote the $v$-th translate of $I$. Further, we introduce the local notation
\begin{equation}
\label{e:gamma}
\gamma_I(x)\coloneqq \left \langle \frac{x -c(I)}{\ell(I)} \right  \rangle^{100}, \qquad x\in \R.
\end{equation}
The polynomial $\gamma_I$ will be used to apply the so-called localization trick. As we perform this a few times in the proof, we isolate the related notation here. If $\T$ is a $k$-lacunary tree and $T_\T$ a tree operator, we write
\begin{equation}
\label{e:tildeT}
\tilde T_{\T} g\coloneqq \sum_{P \in \T}  \tilde S_{P_k} g, \quad \tilde S_{P_k} g  \coloneqq \gamma_{I_\T}  S_{P_k} g.
\end{equation}
It is immediate to verify that $\tilde S_{P_k}\in \mathbb S_{P_k}$ for all $P\in \mathsf T$, so that $\tilde T_\T$ is also a tree operator.

 The proof strategy is an adaptation of \cite[Section 9]{HLC}: indeed, the bound of Lemma \ref{l:engbd}  is an immediate consequence of  the   estimate \eqref{e:engbd1} below. Having fixed a $k$-lacunary tree $\T$, there holds
\begin{equation}
\label{e:engbd1}
\left\| T_{\T} f\right\|_{L^q(\R; \mathcal X)} \lesssim_q \lambda |I_\T|^{\frac1q}, \qquad \lambda\coloneqq \sup_{I\in \mathcal I} \inf_{  I} \mathrm{M}(|f|_\mathcal X), \quad \mathcal I\coloneqq\{I_P: P \in \mathsf T\}, \qquad 1<q<\infty.
\end{equation}
The estimate is uniform over tree operators $T_\T$.

  By modulation invariance of \eqref{e:engbd1},  we may reduce to treating the case $\xi_\T=0$.     Then, estimate \eqref{e:engbd1}  will be obtained as a consequence of the next lemma.
 \begin{lem} \label{l:engbdin} Let $\T$ be a $k$-lacunary tree with $\xi_\T=0$, $T_\T$ be a tree operator. For each $J$-dyadic interval  $K\subset \R$  there exists a constant $a_K$ with the property that
\begin{equation}
\label{e:engbd2}
\left\|\mathbf 1_{K} (T_\T f - a_K)\right\|_{L^{1,\infty}(\R; \mathcal X)} \lesssim \lambda |K|
\end{equation}
with bound independent of $K$, $T_\T$ and $\T$. In particular, if $\ell(K)\geq \ell(I_\T)$ we may take $a_K=0$.
\end{lem}
We use Lemma \ref{l:engbdin} to finish the proof of \eqref{e:engbd1}. Fix a tree operator
$
T_{\T}$. Then, referring to \eqref{e:tildeT},
\begin{equation}
\label{e:engbd24}
\left\| T_{\T} f\right\|_{L^q(\R; \mathcal X)} = \left\| \gamma_{I_\T}^{-1} \tilde T_{\T} f\right\|_{L^q(\R; \mathcal X)} \lesssim \sum_{v\in \mathbb Z} \langle v\rangle^{-100} \left\| \mathbf{1}_{I_\T^{+v}}\tilde T_{\T} f\right\|_{L^q(\R; \mathcal X)}.
\end{equation}
But,  Lemma \ref{l:engbdin} applied to $\tilde T$ together with the John-Str\"omberg inequality yields the two estimates
\[
\left\| \mathbf{1}_{I_\T^{+v}}\tilde T_{\T} f\right\|_{L^{1,\infty}(\R; \mathcal X)} \lesssim \lambda | I_\T|,
\qquad \left\|  \tilde T_{\T} f\right\|_{\mathrm{BMO}(\R; \mathcal X)} \lesssim \lambda,
\]
which together with the John-Nirenberg inequality tell us that
\begin{equation}
\label{e:engbd25}
\left\| \mathbf{1}_{I_\T^{+v}}\tilde T_{\T} f\right\|_{L^q(\R; \mathcal X)} \lesssim \lambda | I_\T|^{\frac1q}.
\end{equation}
A combination of \eqref{e:engbd25} and \eqref{e:engbd24} finally yields \eqref{e:engbd1}.
\begin{proof}[Proof of Lemma \ref{l:engbdin}]
We fix a tree operator and use the local notation
\[
T_\T f= \sum_{I \in \mathcal I} S_I f
\]
where $S_I = S_{P_k}\in \mathbb S_{P_k} $ for the unique tri-tile $P\in \T$ with $I_P=I$.

We begin the proof with the definition of the constant $a_K$. This constant comes from the large scales contribution on $K$, that is the intervals \[\mathcal I^{\mathrm{low}}=\{I \in \mathcal I: \ell(I)>\ell(K)\}.\] For $n \in \mathbb N$ let $K^{(n)}$ be the $n$-th $J$-dyadic parent of $K$. Then if $I\in \mathcal I_{\mathrm{low}}$, it must be $I=K^{(n)+v}$ for some $n\in \mathbb N, v\in \mathbb Z$. We define
\begin{equation}
\label{e:engbd30}
a_K = \sum_{n \geq 1} \sum_{v\in \mathbb Z} S_{K^{(n)+v}} f(c(K))
\end{equation}
where we have simply set $S_{K^{(n)+v}}=0$ if $K^{(n)+v}\not\in \mathcal I$. Clearly, the second claim now follows from the first, as $\mathcal I_{\mathrm{low}}$ is empty, whence $a_K$ is zero, when $\ell(K)\geq \ell(I_\T)$.

We continue with the proof of \eqref{e:engbd2}
We claim that
\begin{equation}
\label{e:engbd3}
\mathbf{1}_K   \sum_{n \geq 1} \sum_{v\in \mathbb Z} \left| S_{K^{(n)+v}} f(c(K)) -S_{K^{(n)+v}} f(x)\right|_{\mathcal X } \lesssim \lambda.
\end{equation}
Indeed, denoting by $F=|f|_{\mathcal X}$, by  $u_{n,v}$ the kernel of $T_{K^{(n)+v}}$ and by $\chi_{n,v}= \chi_{K^{(n)+v}}$ for simplicity, and using the kernel estimates for $u_{n,v}$ and the extra decay in $v$, we have for $x\in K$
\begin{equation}
\label{e:engbd4}
\begin{split} &
\quad \left| S_{K^{(n)+v}} f(c(K)) -S_{K^{(n)+v}} f(x)\right|\\  & \leq \left |\chi_{n,v}(x)-\chi_{n,v}(c(K))\right| \left( F*|u_{n,v}|(x)\right) + \int_{x}^{c(K)}  F *  \left| Du_{n,v}\right|(z)  \, \mathrm{d} z
\\ & \lesssim  \langle v\rangle^{-100} 2^{-n} \inf_{K^{(n)+v}} \mathrm{M}F \leq \langle v\rangle^{-100} 2^{-n} \lambda,
\end{split}
\end{equation}
which is summable over $v,n$ in \eqref{e:engbd3}. The last estimate follows from the membership of $K^{(n)+v}$ to $\mathcal I$.

We now come to the small scales. We first deal with the contribution of the intervals \[\mathcal I^{\mathrm{high}}_{n,v}=\{I\in \mathcal I: \ell(I)=2^{-n} \ell(K), I\subset K^{+v}\}, \qquad n\geq 0, \,v \in \mathbb Z, \,  |v|> 1.\] Notice that this excludes the intervals $
\mathcal I^{\mathrm{high}}=\{I \in \mathcal I: I \subset 3K\} $ which will be handled as the main term.
 The $\mathcal I^{\mathrm{high}}_{n,v}$ are tail terms: in fact, with the same notations as before, if $x\in K$ and $I \in I^{\mathrm{high}}_{n,v}$
\begin{equation}
\label{e:engbd5}
|\chi_{I}(x)| \left( F* |u_{n,v}| (x) \right) \lesssim (v2^n)^{-100} \sum_{t\geq 0} 2^{-100 t} \left\langle  F \right\rangle_{[x-2^{t+1}\ell(I), x+2^{t+1}\ell(I)]}
\end{equation}
As, for $x\in K$,
\[
\left\langle  F \right\rangle_{[x-2^{t+1}\ell(I), x+2^{t+1}\ell(I)]} \lesssim
 \begin{cases} \displaystyle
 \frac{v \ell(K)}{2^{t} \ell(I)}  \left\langle  F \right\rangle_{[x-2^{10}v\ell(K), x+2^{10}v\ell(K)]} \leq (v2^n)  \displaystyle \inf_{x\in I} \mathrm{M} F (x)  &  2^{t+1}\ell(I) \leq v \ell(K) \\
 \\ \displaystyle \inf_{x\in I} \mathrm{M} F (x) &  2^{t+1}\ell(I) > v \ell(K)
\end{cases}
\]
we obtain by summation of \eqref{e:engbd5} that
\begin{equation}
\label{e:engbd6}
\mathbf{1}_K \sum_{n\geq 0} \sum_{|v|\geq 2} \sum_{I \in \mathcal I^{\mathrm{high}}_{n,v}} \left| S_{I} f \right|_{\mathcal X} \lesssim \lambda.
\end{equation}
We are left to estimate the contribution of $\mathcal I^{\mathrm{high}}$. The union of the intervals $\mathcal I^{\mathrm{hi}}$ is contained in $3K$. By possibly splitting   $\mathcal I^{\mathrm{high}}$ into three collections and replacing $I\in \mathcal I^{\mathrm{high}}$ with the corresponding smoothing interval from one of three shifted dyadic grids, so that the union is still contained in $18K$,  we can achieve the property  that if $I,L\in\mathcal I^{\mathrm{high}}$ and  $I\subset L$ then $3I \subset  L$.

 Let now  $L \in\mathcal L$ be the collection  of those $L\in \mathcal I^{\mathrm{high}}$ which are maximal with respect to inclusion and $\mathcal I(L) =\{I\in \mathcal I^{\mathrm{high}}: I\subsetneq L\}$.
First we remove the tops. It is immediate to bound
\begin{equation}
\label{e:engbd7} \sum_{L \in \mathcal L}
\left\| S_{L } f\right\|_{L^1(\R; \mathcal X)} \lesssim   \sum_{L \in \mathcal L} |L| \inf_{L} \mathrm{M}F  \lesssim \lambda |K|.
\end{equation}
We estimate one more tail term. For $n\geq 1$ let $\mathcal I^n(L)=\{I\in \mathcal I(L): \ell (I) = 2^{-n} \ell(L) \}.$
For each $ I \in \mathcal I^n(L)$, let $z_I$ be the least nonnegative integer $z$ such that $(I\pm z\ell(I)) \cap (\R\setminus L )\neq \varnothing$.  As $3I \subset L$, we have $z_I \geq 1$. Furthermore for each  integer $z\geq 1$, there are at most two intervals $ I \in \mathcal I^n(L) $ with $z_I = z$. As for $x \in \R\setminus L$ we have $\dist(x,I) \geq \ell(I)$, there holds
\[
\begin{split}
\mathbf{1}_{\mathbb R\setminus L} (x) |S_I f(x)|_X &
\lesssim \left\langle \textstyle\frac{\dist(x,I)}{\ell(I)} \right\rangle^{-100} \sup_{s\gtrsim \ell(I)} \frac{1}{|B_s(x)|}\int\displaylimits_{B_s(x)} F \lesssim \left\langle \textstyle \frac{\dist(x,I)}{\ell(I)} \right\rangle^{-99} \inf_{I} \mathrm{M} F\\ & \lesssim
\left\langle\textstyle \frac{\dist(x,I)}{\ell(I)} \right\rangle^{-90} z_I^{-9} \lambda
\end{split}
\]
Integrating over $\mathbb R\setminus L$ the last display for each $I$-summand, we have
\[
\sum_{n \geq 1} \sum_{I \in \mathcal I^n(L) }  \left\| \mathbf{1}_{\mathbb R\setminus L}   S_{I } f\right\|_{L^1(\R; \mathcal X)}
\lesssim \lambda \sum_{n \geq 1} \sum_{I \in \mathcal I^n(L) } z_I^{-9} |I| \lesssim \lambda  \sum_{n \geq 1} \sum_{z\geq 1 }z^{-9} 2^{-n}|L|  \lesssim \lambda|L|
\]
whence
\begin{equation}
\label{e:engbd8} \sum_{L \in \mathcal L} \sum_{I\in \mathcal I(L)}  \left\| \mathbf{1}_{\mathbb R\setminus L}  S_{I } f\right\|_{L^1(\R; \mathcal X)} \lesssim \lambda \sum_{L \in \mathcal L} |L| \lesssim  \lambda |K|.
\end{equation}
We are left to estimate the main term. Using disjointness of the supports of the summands below
\begin{equation}
\label{e:engbd9}
\left\| \sum_{L \in \mathcal L} \mathbf{1}_{  L} \sum_{I\in \mathcal I(L)}   S_{I } f\right\|_{L^{1,\infty}(\R; \mathcal X)} \leq
\sum_{L \in \mathcal L} \left\|    \sum_{I\in \mathcal I(L)}   S_{I } f\right\|_{L^{1,\infty}(\R; \mathcal X)}.
\end{equation}
To estimate each summand on the right hand side of the last display, we use the localization trick.  Referring to \eqref{e:gamma}, set $\bar S_I g\coloneqq S_I(\gamma_L g)$. We then have
\begin{equation}
\label{e:engbd10}
\begin{split}&\quad
\sum_{L \in \mathcal L} \left\|    \sum_{I\in \mathcal I(L)}   S_{I } f\right\|_{L^{1,\infty}(\R; \mathcal X)} =
\sum_{L \in \mathcal L} \left\|    \sum_{I\in \mathcal I(L)}  \tilde S_{I }(\gamma_L^{-1} f)\right\|_{L^{1,\infty}(\R; \mathcal X)}\\
&\lesssim  \sum_{L \in \mathcal L} \left\|   \gamma_L^{-1} f\right\|_{L^{1}(\R; \mathcal X)} \lesssim \sum_{L \in \mathcal L} |L|\inf_{L} \mathrm{M}F \lesssim \lambda |K|
\end{split}
\end{equation}
as each $\bar S_I\in \mathbb S_{P}$ where $P \in \mathsf T $ is the unique tri-tile with  $I_P=I$, and therefore each $L$-th summand on the right hand side of the first line is a Calder\'on-Zygmund operator. We achieve \eqref{e:engbd2} by putting together \eqref{e:engbd30}, \eqref{e:engbd3}, \eqref{e:engbd6}, \eqref{e:engbd7}, \eqref{e:engbd8}, \eqref{e:engbd9} and \eqref{e:engbd10}. The proof of the lemma is then complete.
\end{proof}
\bibliography{OP_Multilin_7AUG}

\begin{bibdiv}
\begin{biblist}

\bib{ALV}{article}{
      author={Amenta, Alex},
      author={Lorist, Emiel},
      author={Veraar, Mark},
       title={Fourier multipliers in {B}anach function spaces with {UMD}
  concavifications},
        date={2019},
        ISSN={0002-9947},
     journal={Trans. Amer. Math. Soc.},
      volume={371},
      number={7},
       pages={4837\ndash 4868},
         url={https://doi.org/10.1090/tran/7520},
      review={\MR{3934469}},
}

\bib{AU1}{article}{
      author={{Amenta}, Alex},
      author={{Uraltsev}, Gennady},
       title={{Banach-valued modulation invariant Carleson embeddings and
  outer-$L^p$ spaces: the Walsh case}},
        date={2019May},
     journal={preprint arXiv:1905.08681},
       pages={arXiv:1905.08681},
      eprint={1905.08681},
}

\bib{AU2}{article}{
      author={{Amenta}, Alex},
      author={{Uraltsev}, Gennady},
       title={{The bilinear Hilbert transform on UMD spaces}},
        date={2019May},
     journal={preprint arXiv:1909.06416},
}

\bib{BenMusc15}{article}{
      author={Benea, Cristina},
      author={Muscalu, Camil},
       title={Multiple vector-valued inequalities via the helicoidal method},
        date={2016},
        ISSN={2157-5045},
     journal={Anal. PDE},
      volume={9},
      number={8},
       pages={1931\ndash 1988},
         url={https://doi.org/10.2140/apde.2016.9.1931},
      review={\MR{3599522}},
}

\bib{Bou}{article}{
      author={Bourgain, Jean},
       title={Some remarks on {B}anach spaces in which martingale difference
  sequences are unconditional},
        date={1983},
        ISSN={0004-2080},
     journal={Ark. Mat.},
      volume={21},
      number={2},
       pages={163\ndash 168},
         url={https://doi.org/10.1007/BF02384306},
      review={\MR{727340}},
}

\bib{Burk1}{incollection}{
      author={Burkholder, Donald~L.},
       title={A geometric condition that implies the existence of certain
  singular integrals of {B}anach-space-valued functions},
        date={1983},
   booktitle={Conference on harmonic analysis in honor of {A}ntoni {Z}ygmund,
  {V}ol. {I}, {II} ({C}hicago, {I}ll., 1981)},
      series={Wadsworth Math. Ser.},
   publisher={Wadsworth, Belmont, CA},
       pages={270\ndash 286},
      review={\MR{730072}},
}

\bib{CH}{article}{
      author={Castro, Alejandro~J.},
      author={Hyt\"{o}nen, Tuomas~P.},
       title={Bounds for partial derivatives: necessity of {UMD} and sharp
  constants},
        date={2016},
        ISSN={0025-5874},
     journal={Math. Z.},
      volume={282},
      number={3-4},
       pages={635\ndash 650},
         url={https://doi.org/10.1007/s00209-015-1556-y},
      review={\MR{3473635}},
}

\bib{CrM}{article}{
      author={Cruz-Uribe, David},
      author={Martell, Jos\'{e}~Mar\'{\i}a},
       title={Limited range multilinear extrapolation with applications to the
  bilinear {H}ilbert transform},
        date={2018},
        ISSN={0025-5831},
     journal={Math. Ann.},
      volume={371},
      number={1-2},
       pages={615\ndash 653},
         url={https://doi.org/10.1007/s00208-018-1640-9},
      review={\MR{3788859}},
}

\bib{CDPO}{article}{
      author={Culiuc, Amalia},
      author={Di~Plinio, Francesco},
      author={Ou, Yumeng},
       title={Domination of multilinear singular integrals by positive sparse
  forms},
        date={2018},
        ISSN={0024-6107},
     journal={J. Lond. Math. Soc. (2)},
      volume={98},
      number={2},
       pages={369\ndash 392},
         url={https://doi.org/10.1112/jlms.12139},
      review={\MR{3873113}},
}

\bib{DT}{article}{
      author={Demeter, Ciprian},
      author={Thiele, Christoph},
       title={On the two-dimensional bilinear {H}ilbert transform},
        date={2010},
        ISSN={0002-9327},
     journal={Amer. J. Math.},
      volume={132},
      number={1},
       pages={201\ndash 256},
         url={https://doi.org/10.1353/ajm.0.0101},
      review={\MR{2597511}},
}

\bib{DLMV2}{article}{
      author={Di~Plinio, Francesco},
      author={Li, Kangwei},
      author={Martikainen, Henri},
      author={Vuorinen, Emil},
       title={Multilinear operator-valued {C}alder{\'o}n-{Z}ygmund theory},
        date={2019},
     journal={preprint, arXiv:1908.07233},
}

\bib{DLMV1}{article}{
      author={Di~Plinio, Francesco},
      author={Li, Kangwei},
      author={Martikainen, Henri},
      author={Vuorinen, Emil},
       title={Multilinear singular integrals on non-commutative ${L}^p$
  spaces},
        date={2019},
     journal={preprint, arXiv:1905.02139},
}

\bib{DO}{article}{
      author={Di~Plinio, Francesco},
      author={Ou, Yumeng},
       title={Banach-valued multilinear singular integrals},
        date={2018},
        ISSN={0022-2518},
     journal={Indiana Univ. Math. J.},
      volume={67},
      number={5},
       pages={1711\ndash 1763},
         url={https://doi.org/10.1512/iumj.2018.67.7466},
      review={\MR{3875242}},
}

\bib{GMSS}{article}{
      author={Geiss, Stefan},
      author={Montgomery-Smith, Stephen},
      author={Saksman, Eero},
       title={On singular integral and martingale transforms},
        date={2010},
        ISSN={0002-9947},
     journal={Trans. Amer. Math. Soc.},
      volume={362},
      number={2},
       pages={553\ndash 575},
         url={https://doi.org/10.1090/S0002-9947-09-04953-8},
      review={\MR{2551497}},
}

\bib{HH}{article}{
      author={H\"{a}nninen, Timo~S.},
      author={Hyt\"{o}nen, Tuomas~P.},
       title={Operator-valued dyadic shifts and the {$T(1)$} theorem},
        date={2016},
        ISSN={0026-9255},
     journal={Monatsh. Math.},
      volume={180},
      number={2},
       pages={213\ndash 253},
         url={https://doi.org/10.1007/s00605-016-0891-3},
      review={\MR{3502626}},
}

\bib{HytPor08}{article}{
      author={Hyt{\"o}nen, Tuomas},
      author={Portal, Pierre},
       title={Vector-valued multiparameter singular integrals and
  pseudodifferential operators},
        date={2008},
        ISSN={0001-8708},
     journal={Adv. Math.},
      volume={217},
      number={2},
       pages={519\ndash 536},
         url={http://dx.doi.org/10.1016/j.aim.2007.08.002},
      review={\MR{2370274 (2009e:42027)}},
}

\bib{HNVW1}{book}{
      author={Hyt\"{o}nen, Tuomas},
      author={van Neerven, Jan},
      author={Veraar, Mark},
      author={Weis, Lutz},
       title={Analysis in {B}anach spaces. {V}ol. {I}. {M}artingales and
  {L}ittlewood-{P}aley theory},
      series={Ergebnisse der Mathematik und ihrer Grenzgebiete. 3. Folge. A
  Series of Modern Surveys in Mathematics [Results in Mathematics and Related
  Areas. 3rd Series. A Series of Modern Surveys in Mathematics]},
   publisher={Springer, Cham},
        date={2016},
      volume={63},
        ISBN={978-3-319-48519-5; 978-3-319-48520-1},
      review={\MR{3617205}},
}

\bib{Hyt2005}{article}{
      author={Hyt{\"o}nen, Tuomas~P.},
       title={On operator-multipliers for mixed-norm {$L^{\overline p}$}
  spaces},
        date={2005},
        ISSN={0003-889X},
     journal={Arch. Math. (Basel)},
      volume={85},
      number={2},
       pages={151\ndash 155},
         url={http://dx.doi.org/10.1007/s00013-005-1300-7},
      review={\MR{2161804 (2006c:42011)}},
}

\bib{HLC}{article}{
      author={Hyt\"{o}nen, Tuomas~P.},
      author={Lacey, Michael~T.},
       title={Pointwise convergence of vector-valued {F}ourier series},
        date={2013},
        ISSN={0025-5831},
     journal={Math. Ann.},
      volume={357},
      number={4},
       pages={1329\ndash 1361},
         url={https://doi.org/10.1007/s00208-013-0935-0},
      review={\MR{3124934}},
}

\bib{HLW}{article}{
      author={Hyt\"{o}nen, Tuomas~P.},
      author={Lacey, Michael~T.},
       title={Pointwise convergence of {W}alsh-{F}ourier series of
  vector-valued functions},
        date={2018},
        ISSN={1073-2780},
     journal={Math. Res. Lett.},
      volume={25},
      number={2},
       pages={561\ndash 580},
         url={https://doi.org/10.4310/MRL.2018.v25.n2.a11},
      review={\MR{3826835}},
}

\bib{HLP}{article}{
      author={Hyt\"{o}nen, Tuomas~P.},
      author={Lacey, Michael~T.},
      author={Parissis, Ioannis},
       title={The vector valued quartile operator},
        date={2013},
        ISSN={0010-0757},
     journal={Collect. Math.},
      volume={64},
      number={3},
       pages={427\ndash 454},
         url={https://doi.org/10.1007/s13348-012-0070-3},
      review={\MR{3084406}},
}

\bib{HLPC}{article}{
      author={Hyt\"{o}nen, Tuomas~P.},
      author={Lacey, Michael~T.},
      author={Parissis, Ioannis},
       title={A variation norm {C}arleson theorem for vector-valued
  {W}alsh-{F}ourier series},
        date={2014},
        ISSN={0213-2230},
     journal={Rev. Mat. Iberoam.},
      volume={30},
      number={3},
       pages={979\ndash 1014},
         url={https://doi.org/10.4171/RMI/804},
      review={\MR{3254998}},
}

\bib{JX}{article}{
      author={Junge, Marius},
      author={Xu, Quanhua},
       title={Representation of certain homogeneous {H}ilbertian operator
  spaces and applications},
        date={2010},
        ISSN={0020-9910},
     journal={Invent. Math.},
      volume={179},
      number={1},
       pages={75\ndash 118},
         url={https://doi.org/10.1007/s00222-009-0210-x},
      review={\MR{2563760}},
}

\bib{KW1}{article}{
      author={Kalton, Nigel},
      author={Weis, Lutz},
       title={The {$H^\infty$}-calculus and sums of closed operators},
        date={2001},
        ISSN={0025-5831},
     journal={Math. Ann.},
      volume={321},
      number={2},
       pages={319\ndash 345},
         url={https://doi.org/10.1007/s002080100231},
      review={\MR{1866491}},
}

\bib{LT1}{article}{
      author={Lacey, Michael},
      author={Thiele, Christoph},
       title={{$L^p$} estimates on the bilinear {H}ilbert transform for
  {$2<p<\infty$}},
        date={1997},
        ISSN={0003-486X},
     journal={Ann. of Math. (2)},
      volume={146},
      number={3},
       pages={693\ndash 724},
         url={https://doi.org/10.2307/2952458},
      review={\MR{1491450}},
}

\bib{LT2}{article}{
      author={Lacey, Michael},
      author={Thiele, Christoph},
       title={On {C}alder\'{o}n's conjecture},
        date={1999},
        ISSN={0003-486X},
     journal={Ann. of Math. (2)},
      volume={149},
      number={2},
       pages={475\ndash 496},
         url={https://doi.org/10.2307/120971},
      review={\MR{1689336}},
}

\bib{MTT}{article}{
      author={Muscalu, Camil},
      author={Tao, Terence},
      author={Thiele, Christoph},
       title={Multi-linear operators given by singular multipliers},
        date={2002},
        ISSN={0894-0347},
     journal={J. Amer. Math. Soc.},
      volume={15},
      number={2},
       pages={469\ndash 496},
         url={https://doi.org/10.1090/S0894-0347-01-00379-4},
      review={\MR{1887641}},
}

\bib{MTT2}{incollection}{
      author={Muscalu, Camil},
      author={Tao, Terence},
      author={Thiele, Christoph},
       title={Uniform estimates on multi-linear operators with modulation
  symmetry},
        date={2002},
      volume={88},
       pages={255\ndash 309},
         url={https://doi.org/10.1007/BF02786579},
        note={Dedicated to the memory of Tom Wolff},
      review={\MR{1979774}},
}

\bib{PX}{incollection}{
      author={Pisier, Gilles},
      author={Xu, Quanhua},
       title={Non-commutative {$L^p$}-spaces},
        date={2003},
   booktitle={Handbook of the geometry of {B}anach spaces, {V}ol. 2},
   publisher={North-Holland, Amsterdam},
       pages={1459\ndash 1517},
         url={https://doi.org/10.1016/S1874-5849(03)80041-4},
      review={\MR{1999201}},
}

\bib{Silva12}{article}{
      author={Silva, Prabath},
       title={Vector-valued inequalities for families of bilinear {H}ilbert
  transforms and applications to bi-parameter problems},
        date={2014},
        ISSN={0024-6107},
     journal={J. Lond. Math. Soc. (2)},
      volume={90},
      number={3},
       pages={695\ndash 724},
         url={http://dx.doi.org/10.1112/jlms/jdu044},
      review={\MR{3291796}},
}

\bib{SteinB}{book}{
      author={Stein, Elias~M.},
       title={Harmonic analysis: real-variable methods, orthogonality, and
  oscillatory integrals},
      series={Princeton Mathematical Series},
   publisher={Princeton University Press},
     address={Princeton, NJ},
        date={1993},
      volume={43},
        ISBN={0-691-03216-5},
        note={With the assistance of Timothy S. Murphy, Monographs in Harmonic
  Analysis, III},
      review={\MR{1232192 (95c:42002)}},
}

\bib{ThWPA}{book}{
      author={Thiele, Christoph},
       title={Wave packet analysis},
      series={CBMS Regional Conference Series in Mathematics},
   publisher={Published for the Conference Board of the Mathematical Sciences,
  Washington, DC; by the American Mathematical Society, Providence, RI},
        date={2006},
      volume={105},
        ISBN={0-8218-3661-7},
         url={https://doi.org/10.1090/cbms/105},
      review={\MR{2199086}},
}

\bib{Weis}{article}{
      author={Weis, Lutz},
       title={Operator-valued {F}ourier multiplier theorems and maximal
  {$L_p$}-regularity},
        date={2001},
        ISSN={0025-5831},
     journal={Math. Ann.},
      volume={319},
      number={4},
       pages={735\ndash 758},
         url={https://doi.org/10.1007/PL00004457},
      review={\MR{1825406}},
}

\end{biblist}
\end{bibdiv}
\bibliographystyle{amsplain}
\end{document}